\documentclass[11pt]{amsart}

\pdfoutput=1

\usepackage[text={420pt,660pt},centering]{geometry}

\usepackage{cancel}
\usepackage{esint,amssymb}
\usepackage{graphicx}
\usepackage{mathtools} 
\usepackage[colorlinks=true, pdfstartview=FitV, linkcolor=blue, citecolor=blue, urlcolor=blue,pagebackref=false]{hyperref}
\usepackage{microtype}

\usepackage{bm}
\usepackage{dsfont}
\usepackage{mathrsfs}
\usepackage{xcolor}
\usepackage{accents} \usepackage{enumitem} 

\usepackage[normalem]{ulem}

\newtheorem{proposition}{Proposition}
\newtheorem{theorem}[proposition]{Theorem}
\newtheorem{lemma}[proposition]{Lemma}
\newtheorem{corollary}[proposition]{Corollary}

\theoremstyle{remark}
\newtheorem{remark}[proposition]{Remark}

\theoremstyle{definition}
\newtheorem{definition}[proposition]{Definition}

\numberwithin{equation}{section}
\numberwithin{proposition}{section}
\numberwithin{figure}{section}
\numberwithin{table}{section}

\newcommand{\N}{\mathbb{N}}

\newcommand{\R}{\mathbb{R}}

\newcommand{\E}{\mathbb{E}}
\renewcommand{\P}{\mathbb{P}}

\renewcommand{\S}{\mathbf{S}}

\newcommand{\basis}{\Sigma}

\newcommand{\std}{{\mathrm{std}}}

\newcommand{\identity}{{\mathbf{Id}}}
\newcommand{\bi}{{\mathbf{i}}}
\newcommand{\bal}{{\boldsymbol{\alpha}}}
\newcommand{\brho}{{\boldsymbol{h}}}
\newcommand{\one}{{\boldsymbol{1}}}
\newcommand{\bh}{{\boldsymbol{h}}}
\newcommand{\bg}{{\boldsymbol{g}}}
\renewcommand{\zeta}{{\alpha^{-1}}}

\newcommand{\eps}{\varepsilon}

\renewcommand{\leq}{\leqslant}
\renewcommand{\geq}{\geqslant}

\renewcommand{\subset}{\subseteq}
\renewcommand{\bar}{\overline}
\renewcommand{\tilde}{\widetilde}

\renewcommand{\hat}{\widehat}

\newcommand{\Ll}{\left}
\newcommand{\Rr}{\right}
\renewcommand{\d}{\mathrm{d}}

\newcommand{\D}{D}
\newcommand{\sP}{\mathscr{P}}
\newcommand{\m}{m}

\DeclareMathOperator{\tr}{tr}
\DeclareMathOperator{\supp}{supp}

\newcommand{\la}{\left\langle}
\newcommand{\ra}{\right\rangle}

\newcommand{\sF}{{\mathscr{F}}}
\newcommand{\cM}{\mathcal{M}}
\newcommand{\rhs}{\mathrm{RHS}}
\newcommand{\lhs}{\mathrm{LHS}}
\newcommand{\cleq}{\preccurlyeq}
\newcommand{\cgeq}{\succcurlyeq}
\renewcommand*{\dot}[1]{\accentset{\mbox{\large\bfseries .}}{#1}}
\renewcommand{\gamma}{\dot\mu}
\newcommand{\Sym}{{\mathrm{Sym}(\D)}}
\newcommand{\ks}{{\mathbf{s}}}

\begin{document}

\author{Hong-Bin Chen}
\address{Institut des Hautes \'Etudes Scientifiques, France}
\email{hbchen@ihes.fr}


\title[Parisi PDE and convexity]{Parisi PDE and convexity for vector spins}

\begin{abstract}
We consider mean-field vector spin glasses with self-overlap correction. The limit of free energy is known to be the Parisi formula, which is an infimum over matrix-valued paths. We decompose such a path into a Lipschitz matrix-valued path and the quantile function of a one-dimensional probability measure. For such a pair, we associate a Parisi PDE generalized for vector spins. Under mild conditions, we rewrite the Parisi formula in terms of solutions of the PDE. Moreover, for each fixed Lipschitz path, the Parisi functional is strictly convex over probability measures.
\end{abstract}

\maketitle

\section{Introduction}

\subsection{Setting}
Fix an integer $\D$ and let $P_1$ be a finite measure on $\R^\D$. For each $N\in\N$, we sample $\sigma = (\sigma_{ki})_{1\leq k\leq \D,\,1\leq i\leq N}\in\R^{\D\times\N}$ by independently drawing the column vectors $\sigma_{\bullet i} = (\sigma_{ki})_{1\leq k\leq \D}$ from $P_1$ for every $i\in\{1,\dots,N\}$. The distribution $P_N$ of $\sigma$ is thus given by $\d P_N(\sigma) = \otimes_{i=1}^N\d P_1(\sigma_{\bullet i})$.
We interpret $P_1$ as the distribution for a single $\R^\D$-valued spin and $\sigma$ as the spin configuration consisting of $N$ independent spins.
We assume that
\begin{gather}
    \text{$P_1$ is supported on the unit ball of $\R^\D$;}\label{e.supp_P_1}
    \\
    \text{$P_1$ is not a Dirac measure.}\label{e.not_Dirac}
\end{gather}

Let $\S^D\subset \R^{\D\times \D}$ be the set of $D\times D$ symmetric matrices. Let $\S^\D_+\subset \S^\D$ (resp.\ $\S^\D_{++}$) be the subset consisting of positive semi-definite (resp.\ definite) matrices.

For two vectors or matrices, $a$ and $b$, of the same dimension, we write $a\cdot b = \sum_{i,j}a_{ij}b_{ij}$ as the Frobenius (entry-wise) inner product and $|a|=\sqrt{a\cdot a}$. Sometimes, for $x,y\in \R^\D$ and $a,b\in \S^\D$, we write $\la x,y\ra_{\R^\D} = x\cdot y$ and  $\la a,b\ra_{\S^\D}= a\cdot b$ for clarity of the ambient space.
For $a, b\in \R^{\D\times \D}$, we write $a\cleq b$ and $b\cgeq a$ if $a\cdot c\leq b\cdot c$ for all $c\in\S^\D_+$. 
For $a\in\S^\D_+$, we denote its matrix square root by $\sqrt{a}$.

For each $N$, we are given a centered Gaussian process $(H_N(\sigma))_{\sigma\in\R^{\D\times N}}$ with covariance
\begin{align}\label{e.xi}
    \E H_N(\sigma)H_N(\sigma') = N\xi\Ll(\frac{\sigma\sigma'^\intercal}{N}\Rr), \quad\forall \sigma,\sigma'\in\R^{\D\times N}
\end{align}
for $\xi:\R^{\D\times\D}\to \R$ satisfying conditions \ref{i.xi_loc_lip}--\ref{i.xi_incre} stated in Section~\ref{s.prelim}.

We are interested in describing the limit of the (standard) free energy:
\begin{align*}
    F^\std_N = \frac{1}{N}\E\log \int \exp\Ll(H_N(\sigma)\Rr)\d P_N(\sigma)
\end{align*}
as $N\to\infty$.
Compared with the Sherrington--Kirkpatrick (SK) model ($\D=1$, $P_1$ uniform on $\{-1,+1\}$, and $\xi(r)=\beta r^2$), one difficulty of the general setting is that the $\S^\D_+$-valued \textit{self-overlap} $\frac{\sigma\sigma^\intercal}{N}$ is not constant.

To bypass this difficulty, we focus on a modification of $F^\std_N$:
\begin{align*}
    F_N = \frac{1}{N}\E \log \int \exp\Ll(H_N(\sigma)-\frac{N}{2}\xi\Ll(\frac{\sigma\sigma^\intercal}{N}\Rr)\Rr)\d P_N(\sigma).
\end{align*}
We call the additional term the \textit{self-overlap correction}, which is $-\frac{1}{2}$ times the variance of $H_N(\sigma)$. 
If $\xi$ is convex on $\S^\D_+$, then we have the Parisi formula
\begin{align}\label{e.F_N=inf}
    \lim_{N\to\infty}F_N = \inf_{\pi\in\Pi} \sP(\pi)
\end{align}
with infimum over
\begin{align}\label{e.Pi}
    \Pi = \Ll\{\pi:[0,1]\to\S^\D_+\ \big|\  \text{$\pi$ is left-continuous and increasing}\Rr\}
\end{align}
where $\pi$ is said to be increasing if $\pi(s)\cgeq \pi(s')$ for $s\geq s'$. The functional $\sP(\pi)$ is defined in terms of Ruelle probability cascades. We postpone the definition to Section~\ref{s.pf_main}. The Parisi formula~\eqref{e.F_N=inf} was proved in \cite[Corollary~8.3]{HJ_critical_pts} (also see \cite[Remark~8.5]{HJ_critical_pts}), which uses the upper bound via the Hamilton--Jacobi equation approach \cite{mourrat2020nonconvex,mourrat2020free,chen2022hamilton,chen2022hamilton2}. Under a stronger assumption that $\xi$ is convex on $\R^{\D\times\D}$, \eqref{e.F_N=inf} was proved in \cite[Theorem~1.1]{chen2023self} via more classical spin glass techniques.

In fact, the infimum in \eqref{e.F_N=inf} can be taken over paths with a fixed endpoint. By \cite[Theorem~1.1]{chen2023on} (with $x=0$ therein), if $\xi$ is convex on $\S^\D_+$, then there is $z\in\S^\D_+$ such that
\begin{gather}
    \lim_{N\to\infty} \E\la \Ll|\frac{\sigma\sigma^\intercal}{N}-z\Rr|\ra =0, \label{e.z}
    \\
    \label{e.parisi}
    \lim_{N\to\infty}F_N = \inf_{\pi \in \Pi(z)} \sP(\pi),
\end{gather}
where $\la\cdot\ra$ is the Gibbs measure associated with $F_N$ and
\begin{align}\label{e.Pi(z)}
    \Pi(z) = \Ll\{\pi\in\Pi\ \big| \  \pi(1) = z\Rr\}.
\end{align}

\subsection{Motivation and main results}
A natural question is whether there is a unique minimizer.
One approach is to prove the strict convexity of the functional.
In this work, we offer a partial answer.
Our approach starts by rewriting the Parisi formula in terms of Parisi PDE solutions.

In the initial prediction by Parisi \cite{parisi80}, the limit of free energy $F_N^\mathrm{SK}$ in the SK model is expressed in terms of a parabolic PDE with one spatial dimension, which is later referred to as the Parisi PDE. In a slight more general model, the mixed $p$-spin model with Ising spins ($D=1$, $P_1$ uniform on $\{-1,+1\}$, and $\xi(r) = \sum_{p=2}^\infty \beta^2_p r^p$), the Parisi PDE takes the form
\begin{align*}
    \partial_s u_\alpha(s,x)+ \frac{1}{2}\xi''(s)\Ll(\partial_x^2 u_\alpha(s,x) + \alpha(s)\Ll(\partial_x^2 u_\alpha(s,x)\Rr)^2\Rr) = 0,\quad (s,x) \in [0,1]\times \R
\end{align*}
with terminal condition $u_\alpha(1,x) = \log\cosh(x)$,
where $\alpha$ is the cumulative distribution function of a probability measure on $[0,1]$. Hence, $\alpha$ belongs to
\begin{align}\label{e.cM=}
    \cM =\Ll\{\alpha:[0,1]\to[0,1]\ \big|\ \text{$\alpha$ is right-continuous and increasing; $\alpha(0)=0$; $\alpha(1)=1$}\Rr\}.
\end{align}
It is well-known that
\begin{align*}
    \lim_{N\to\infty} F_N^\text{mixed-$p$} = \inf_{\alpha\in\cM} \Ll(u_\alpha(0,0)-\frac{1}{2}\int_0^1r\xi''(r)\alpha(r)\d r\Rr).
\end{align*}
In our setting of vector spins, there is an more infimum to take.
Let us explain this below.

Heuristically, if $\pi$ is a minimizer of \eqref{e.parisi}, then the limit $R_\infty$ of the \textit{overlap} $\frac{\sigma\sigma'^\intercal}{N}$ has the law of $\pi(U)$ where $U$ is the uniform random variable on $[0,1]$.
Panchenko's synchronization states that $R_\infty = \Psi\Ll(\tr(R_\infty)\Rr)$ for some $\Psi$ in
where
\begin{align}\label{e.Pi_Lip(z)}
    \Pi_\mathrm{Lip}(z)=\Ll\{\Psi\in\Pi(z)\ \big|\ \text{$\Psi$ is Lipschitz}\Rr\}.
\end{align}
Let $\alpha$ be the cumulative distribution function of $\tr(R_\infty)$ and $\alpha^{-1}$ be its left-continuous inverse (known as the \textit{quantile function}) defined by
\begin{align*}\alpha^{-1}(s) = \inf\{t\in[0,1]:s\leq \alpha(t)\},\quad\forall s\in [0,1].
\end{align*}
Since $\tr(R_\infty)$ has the law of $\alpha^{-1}(U)$, we get $\pi = \Psi\circ \alpha^{-1}$.
Hence, $\pi$ is encoded in two pieces of information: a Lipschitz path $\Psi$ and a distribution $\alpha$ along it.
In view of \eqref{e.parisi}, we expect an infimum over $\alpha$ and another over $\Psi$.

We describe the Parisi PDE in our setting.
For $\Psi\in\Pi_\mathrm{Lip}(z)$, we set $\mu:[0,1]\to \S^\D$ by
\begin{align}\label{e.mu}
    \mu(s) = \nabla\xi\circ\Psi(s),\quad\forall s\in[0,1]
\end{align}
and write $\gamma(s) = \frac{\d}{\d s}\mu(s)$.
(By the local Lipschitzness of $\nabla\xi$ assumed in \ref{i.xi_loc_lip} and Rademacher's theorem, $\gamma$ exists a.e.)
The Parisi PDE generalized for vector spins is the following:
\begin{gather}\label{e.Parisi_PDE}
    \partial_s \Phi(s,x) + \frac{1}{2}\la \gamma(s)\,,\,\nabla^2\Phi(s,x)+\alpha(s)(\nabla \Phi(s,x))(\nabla \Phi(s,x))^\intercal\ra_{\S^D}  =0, 
\end{gather}
for $(s,x)\in [0,1]\times \R^\D$, 
together with the terminal condition
\begin{align}\label{e.Phi(1,x)=_intro}
    \Phi(1,x) = \log \int \exp\Ll( x\cdot \sigma\Rr)\d \tilde P_1^z(\sigma), \quad x\in \R^\D,
\end{align}
where
$\tilde P_1^z$ is a tilting of $P_1$ given by
\begin{align}\label{e.tildeP_1}
    \d \tilde P_1^z(\sigma) = \exp\Ll(-\frac{1}{2}\nabla\xi(z)\cdot\sigma\sigma^\intercal \Rr) \d P_1(\sigma) = \exp\Ll(-\frac{1}{2}\mu(1)\cdot\sigma\sigma^\intercal \Rr) \d P_1(\sigma).
\end{align}
We denote the solution by $\Phi_{\tilde P_1^z,\Psi,\alpha}$, which is roughly defined as follows (see Definitions~\ref{d.sol_pde_discrete} and~\ref{d.sol_pde}). If $\alpha$ is a step function, we solve the equation by the Cole--Hopf transform on each interval where $\alpha$ is constant. For general $\alpha$, we define $\Phi_{\tilde P_1^z,\Psi,\alpha}$ to be the limit of $\Phi_{\tilde P_1^z,\Psi,\alpha_n}$ for step functions $(\alpha_n)_{n\in\N}$ approximating $\alpha$.
We will give sufficient conditions for the existence of the solution in Corollary~\ref{c.well-defined}.

Our first result is a rewriting of the Parisi formula~\eqref{e.parisi}.
Define $\theta:\R^{\D\times\D}\to\R$ by
\begin{align}\label{e.theta}
    \theta(a) = a\cdot\nabla\xi(a)-\xi(a),\quad\forall a \in \R^{\D\times \D}.
\end{align}
For $\Psi\in \Pi_\mathrm{Lip}(z)$ and $\alpha\in\cM$, provided that $\Phi_{\tilde P_1^z,\Psi,\alpha}$ exists, we define
\begin{align}\label{e.sF}
    \sF(\Psi,\alpha) = \E\Ll[\Phi_{\tilde P_1^z,\Psi,\alpha}\Ll(0,\sqrt{\mu(0)}\eta\Rr)\Rr]+ \frac{1}{2}\theta(z) - \frac{1}{2}\int_0^1\alpha(s)\Psi(s)\cdot
    \gamma(s) \d s
\end{align}
where $\eta$ is the standard $\R^\D$-valued Gaussian vector and $\E$ integrates $\eta$.
We need to consider $\Psi$ with additional regularity and set 
\begin{align}\label{e.Pi_reg}
    \Pi_\mathrm{reg}(z) =\Ll\{\Psi\in \Pi_\mathrm{Lip}(z):\: \text{$\Psi$ is continuously differentiable};\; \frac{\d}{\d s}\Psi(s)\in\S^\D_{++},\,\forall s\in[0,1]\Rr\}.
\end{align}

\begin{theorem}
\label{t.lim}
Suppose $z\in \S^\D_{++}$ and that $\xi$ is twice continuously differentiable satisfying
\begin{align}\label{e.cond_xi}
    c\cdot\nabla(b\cdot\nabla\xi)(a)>0,\qquad\forall a \in\S^\D_+,\  b\in\S^\D_+\setminus\{0\},\ c\in\S^\D_{++}.
\end{align}
If $\xi$ is convex on $\S^\D_+$, then
\begin{align*}
    \lim_{N\to\infty} F_N = \inf_{\Psi\in\Pi_\mathrm{reg}(z)}\inf_{\alpha\in\cM} \sF(\Psi,\alpha).
\end{align*}
\end{theorem}

We view \eqref{e.cond_xi} as a monotonicity condition on $\nabla \xi$. 
which is needed to ensure the existence of $\Phi_{\tilde P_1^z,\Psi,\alpha}$ for every $\Psi \in \Pi_\mathrm{reg}(z)$.
We will show (in Lemma~\ref{l.sF=sP}) $\sF(\Psi,\alpha) = \sP(\pi)$ with $\Psi\circ\alpha^{-1} = \pi$.

In the spirit of Auffinger and Chen's work \cite{aufche}, we have a result on the convexity of $\sF$.

\begin{theorem}\label{t.convex}
If $\Psi\in \Pi_\mathrm{Lip}(z)$ satisfies that
\begin{align}\label{e.gamma>0}
    \text{$\gamma$ exists everywhere on $[0,1]$ and $\gamma(s)\in\S^\D_{++}$ for almost every $s\in[0,1]$}
\end{align}
with respect to the Lebesgue measure, then the following holds:
\begin{enumerate}
    \item\label{i.Phi_convex} the functional $\cM\ni \alpha\mapsto \sF(\Psi,\alpha)$ is strictly convex: if $\alpha_0,\,\alpha_1\in\cM$ are distinct, then
    \begin{align*}
        \sF\Ll(\Psi,(1-\lambda)\alpha_0+\lambda\alpha_1\Rr) < (1-\lambda)\sF(\Psi,\alpha_0)+\lambda \sF(\Psi,\alpha_1),\quad\forall \lambda\in(0,1);
    \end{align*}
    \item\label{i.uniq_minimizer} there is a unique minimizer $\alpha^\star$ of $\inf_{\alpha\in\cM}\sF(\Psi,\alpha)$.
\end{enumerate}
\end{theorem}
In particular, condition~\eqref{e.gamma>0} guarantees the existence of $\Phi_{\tilde P_1^z,\Psi,\alpha}$.

\begin{remark}
The convexity of $\xi$ is needed in Theorem~\ref{t.lim} (also in Theorem~\ref{t.D=1}) to have~\eqref{e.parisi}. Hence, we can replace the convexity assumption therein by~\eqref{e.parisi}. Theorem~\ref{t.convex} does not require $\xi$ to be convex.
\end{remark}

\subsection{Discussion on applications}\label{s.discussion}

The most interesting situation is when we can find some $\Psi^\star$ such that
\begin{align}\label{e.limF_N=inf_alpha}
    \lim_{N\to \infty} F_N = \inf_{\alpha\in\cM}\sF(\Psi^\star,\alpha).
\end{align}
Then, we hope to apply Theorem~\ref{t.convex} to get the strict convexity of $\alpha\mapsto \sF(\Psi^\star,\alpha)$
and the uniqueness of the minimizer. 
It seems that the identification of such $\Psi^\star$ depends on the specific setup and does not follow easily from Theorem~\ref{t.lim}.

When $\D=1$, we will show in Theorem~\ref{t.D=1} that \eqref{e.limF_N=inf_alpha} holds for $\Psi^\star: s\mapsto zs$. 
Moreover, the strict convexity and the uniqueness follow from Theorem~\ref{t.convex}.

For $\D> 1$, it is less obvious to find such $\Psi^\star$ in general. 
Nevertheless, a particularly promising setting is the Potts spin glass model with symmetric spin interactions. To be explained in Section~\ref{s.Psi_potts}, we expect
\begin{align*}
    \Psi^\star(s) = \frac{s}{\D}\identity_\D + \frac{1-s}{\D^2}\one_\D,\quad s\in[0,1]
\end{align*}
where $\identity_\D$ is the $\D\times \D$ identity matrix and $\one_\D$ is the $\D\times\D$ matrix with all entries equal to $1$.
A rigorous verification of this for the constraint free energy (different from our modification) can be seen in \cite{bates2023parisi}.
We consider the mixed $p$-spin interaction and let $(\beta_p)_{p=2}^\infty$ be the sequence of inverse temperatures for the $p$-th order interaction ($\beta_p\geq 0$). Although we do not prove \eqref{e.limF_N=inf_alpha}, we will show in Theorem~\ref{t.convex_Potts} the strict convexity of $\alpha\mapsto \sF(\Psi^\star,\alpha)$ and the uniqueness of the minimizer of $\inf_{\alpha\in\cM}\sF(\Psi^\star,\alpha)$. More precisely, we show the following:
\begin{itemize}
    \item If $\beta_p>0$ for some $p\geq 3$, then~\eqref{e.gamma>0} holds and Theorem~\ref{t.convex} is applicable.
    \item When the interaction is purely quadratic (i.e.\ $\beta_2>0$ and $\beta_p=0$ for all $p\geq 3$), unfortunately,~\eqref{e.gamma>0} does not hold. Interestingly, we can still exploit properties of $\mu$ to prove the desired results.
\end{itemize}

\subsection{Related works}

The Parisi formula was initially proposed in \cite{parisi79,parisi80}  for the SK model, which was later mathematically established in \cite{gue03,Tpaper}.
Generalizations were made in various settings: the SK model with soft spins \cite{pan05} and spherical spins \cite{tal.sph}, the mixed $p$-spin model with Ising spins \cite{panchenko2014parisi} and spherical spins \cite{chen2013aizenman}, the multi-species model with Ising spins \cite{pan.multi} and spherical spins \cite{bates2022free},
the Potts spin glass \cite{pan.potts},
the mixed $p$-spin model with vector spins \cite{pan.vec}, and an extension along a Hamilton--Jacobi equation \cite{mourrat2020extending}.
The self-overlap correction appeared in the Hamilton--Jacobi approach by Mourrat \cite{mourrat2019parisi,mourrat2020extending,mourrat2020nonconvex,mourrat2020free} and was recently revisited in \cite{chen2023self} for vector spin glasses.
For more detail on the Hamilton--Jacobi equation approach to statistical mechanics, we refer to~\cite{HJbook}.

The Parisi PDE has been widely studied, mainly in the SK model. Quantitative studies of the solution have led to important results about the Parisi formula \cite{talagrand2006parisi,auffinger2015properties,aufche,talagrand2006parisi}. It was predicted in \cite{mezard1987spin} that the minimizer of the Parisi formula in the SK model is unique. The question about the strict convexity of the Parisi functional was raised in \cite{panchenko2005question}, which has been studied in \cite{panchenko2005question,talagrand2006parisi,Tpaper,bovklim,chen2015partial,jagtob} and was resolved in \cite{aufche} through a variational representation of the Parisi PDE solution.
In the setting of the SK model, \cite{panchenko2005question,talagrand2006parisi,Tpaper,chen2015partial} studied the Parisi PDE using the classical Cole--Hopf transformation together with probability methods and \cite{jagtob} presented a PDE approach.
The multi-dimensional Parisi PDE was considered in \cite{bovklim} to describe the Parisi functional in the vector spin glass with quadratic interaction. There, the PDE was solved in the viscosity sense and a variational representation of the solution was also obtained. 
In this work, we adopt the approach of the Cole--Hopf transformation and probabilistic methods.

Recently, \cite{bates2023parisi} has shown that the limit of the constraint free energy with symmetric interaction in the Potts spin glass is of the form in~\eqref{e.limF_N=inf_alpha} (see \cite[Theorem~1.7 and Remark~1.8]{bates2023parisi}) for $\Psi^\star$ described above. It is remarked in \cite{bates2023parisi} that similar arguments in \cite{aufche} can be extended to show the uniqueness of the minimizer. We believe that our results (Theorems~\ref{t.convex} and~\ref{t.convex_Potts}) are applicable.

\subsection{Organization}\label{s.org}

In Section~\ref{s.PDE}, we study the solutions of the Parisi PDE. We start by solving the PDE using the Cole–Hopf transformation for discrete $\alpha$. Then, we use the probabilistic representation of the solution in terms of the Ruelle probability cascade to show that the solution is Lipschitz in $\alpha$. This allows us to define the solution for general $\alpha$ through continuity.
We finish by proving regularity properties of the solution.

In Section~\ref{s.convex}, we first establish a variational formula for the Parisi PDE solution. Then, we investigate the convexity of the solution in $x\in\R^\D$. This enables us to show that the variational formula for the solution has a unique maximizer. Finally, we show that the solution is convex in $\alpha$.

In Section~\ref{s.pf_main}, we collect results in previous sections to prove the main results. We also prove the result in one dimension, which was mentioned previously in Section~\ref{s.discussion}.

In Section~\ref{s.Potts}, we introduce the setting of the Potts spin glass and prove the results mentioned in Section~\ref{s.discussion}.

\subsection{Acknowledgements}
The author is grateful to Erik Bates from whom the author learned the idea to determine the synchronization map by using symmetry in the Potts spin glass. This project has received funding from the European Research Council (ERC) under the European Union’s Horizon 2020 research and innovation programme (grant agreement No.\ 757296).

\section{Parisi PDE}\label{s.PDE}

In the section, we study the Parisi PDE generalized to the vector spin setting and prove useful properties of its solution.
We refer to Section~\ref{s.org} for the outline of this section.

\subsection{Preliminaries}\label{s.prelim}

Throughout this work, we assume that $\xi$ in \eqref{e.xi} satisfies
\begin{enumerate}[start=1,label={\rm (H\arabic*)}]
    \item \label{i.xi_loc_lip}
    $\xi$ is differentiable and $\nabla\xi$ is locally Lipschitz;
    \item $\xi\geq 0$ on $\S^\D_+$, $\xi(0)=0$, and $\xi(a)=\xi(a^\intercal)$ for all $a\in \R^{\D\times\D}$;
    \item \label{i.xi_incre} if $a,\,b\in\S^\D_+$ satisfies $a\cgeq b$, then $\xi(a)\geq \xi(b)$ and $\nabla \xi(a)\cgeq \nabla \xi(b)$.
\end{enumerate}
We recall following fact \cite[Theorem 7.5.4]{horn2012matrix}:
\begin{align}\label{e.ab>0}
    \text{for $a\in\S^\D$,\qquad $a\in\S^\D_+$ if and only if $a\cdot b\geq 0$ for every $b\in \S^\D_+$.}
\end{align}
Therefore, if $a\in\S^\D$, then $a\in\S^\D_+$ if and only if $a\cgeq 0$.

We equip $\cM$ in \eqref{e.cM=} with metric
\begin{align}\label{e.d_M=}
    d_\cM\Ll(\alpha,\alpha'\Rr) = \int\Ll|\alpha(r) - \alpha'(r)\Rr|\d r,\quad\forall \alpha,\alpha'\in \cM.
\end{align}
We describe the space of discrete cumulative distribution functions.
Let $\cM_\d\subset \cM$ consist of $\alpha:[0,1]\to[0,1]$ of the following form: for some $K\in\N$,
\begin{align}\label{e.alpha=}
    \alpha = \sum_{l=0}^{K}(\m_{l}-\m_{l-1})\mathds{1}_{[q_l,\infty)} = \sum_{l=0}^{K-1} m_l \mathds{1}_{[q_l,q_{l+1})}  + m_K \mathds{1}_{\{q_K\}}
\end{align}
where $\m_{-1},m_0,\dots,\m_{K}\in[0,1]$ and $q_0,q_1,\dots,q_K\in [0,1]$ satisfy
\begin{align}\label{e.m,q}
    0=\m_{-1}\leq\m_0< \dots<\m_{K} = 1,\qquad 
    0= q_0\leq q_1<\dots< q_K=1.
\end{align}
It is straightforward that $\cM_\d $ is a dense subset of $\cM$.

For $\alpha$ given in~\eqref{e.alpha=}, its left-continuous inverse $\alpha^{-1}:[0,1]\to[0,1]$ is given by
\begin{align}\label{e.alpha^-1=}
    \alpha^{-1} = q_0 \mathds{1}_{\{m_{-1}\}} + \sum_{l=0}^K q_l \mathds{1}_{(m_{l-1},m_l]}.
\end{align}
We work with the following condition: 
\begin{align}\label{e.weak_gamma>0}
    \text{$ \frac{\d}{\d s}\sqrt{\mu(t)-\mu(s)} $ exists for every $s,t$ satisfying $0\leq s<t\leq 1$.}
\end{align}
We will show in Lemma~\ref{l.weak->gamma>0} that~\eqref{e.weak_gamma>0} is satisfied under conditions in Theorems~\ref{t.lim} and~\ref{t.convex}. We state a simple fact.

\begin{lemma}\label{e.d+d=-gamma}
If $\Psi$ satisfies~\eqref{e.weak_gamma>0}, then, for $0\leq s <t\leq 1$, \begin{align*}
    \Ll(\frac{\d}{\d s}\sqrt{\mu(t)-\mu(s)}\Rr)\sqrt{\mu(t)-\mu(s)} + \sqrt{\mu(t)-\mu(s)}\Ll(\frac{\d}{\d s}\sqrt{\mu(t)-\mu(s)}\Rr) = -\gamma(s).
\end{align*}
\end{lemma}
The left-hand side of the above display is a sum of two matrix multiplications.
\begin{proof}
Write $\bar \mu(s)=\mu(t)-\mu(s)$.
Since
\begin{align*}
    \bar\mu(s+\eps) = \Ll(\sqrt{\bar\mu(s+\eps)}\Rr)^2 = \Ll(\sqrt{\bar\mu(s)}+\eps \Ll(\sqrt{\bar \mu}\Rr)'(s) + o(\eps)\Rr)^2 
    \\
    = \bar\mu(s) +\eps \Ll(\sqrt{\bar\mu(s)}\sqrt{\bar \mu}'(s) + \sqrt{\bar \mu}'(s)\sqrt{\bar\mu(s)}\Rr)+o(\eps)
\end{align*}
and $\bar\mu(s+\eps) = \bar\mu(s) - \eps \gamma(s)+o(\eps)$,
sending $\eps \to0$, we obtain the announced identity.
\end{proof}

We will fix the initial condition to be
\begin{align}\label{e.phi=initial}
    \phi(x) = \log\int \exp\left(x\cdot\sigma\right) \d P_1(\sigma),\quad\forall x\in \R^\D.
\end{align}
The initial condition in~\eqref{e.Phi(1,x)=_intro} corresponds to the above with $\tilde P_1^z$ in~\eqref{e.tildeP_1} substituted for $P_1$. 
We only need this substitution later in Section~\ref{s.pf_main}.
Here, for generality, we work with~\eqref{e.phi=initial}.

Throughout, the gradient operator $\nabla =(\partial_{x_i})_{i=1}^\D$ and the Hessian operator $\nabla^2 = (\partial^2_{x_i,x_j})_{i,j=1}^\D$ take derivatives in $x \in \R^\D$. For a differentiable function $g:\R^\D\to\R$, $\nabla g$ takes value in $\R^\D$ and $\nabla^2g$ takes value in $\S^\D$.

\begin{remark}\label{r.condition_PDE}
In this section, we only assume that $P_1$ satisfies~\eqref{e.supp_P_1}. Condition~\eqref{e.not_Dirac} is not needed. For $\Psi \in \Pi_\mathrm{Lip}(z)$, we assume~\eqref{e.weak_gamma>0} instead of conditions in the main theorems. This choice makes the Parisi PDE solution well-defined in the Potts model with quadratic interaction (see Lemma~\ref{l.weak_cond_Potts2}).
\end{remark}

\subsection{PDE for discrete paths}

\begin{definition}[Parisi PDE solution for $\alpha\in\cM_\d$]\label{d.sol_pde_discrete}

For $P_1$ satisfying~\eqref{e.supp_P_1}, $\Psi\in\Pi_\mathrm{Lip}(z)$ satisfying~\eqref{e.weak_gamma>0}, and $\alpha\in\cM_\d$ of the form~\eqref{e.alpha=}, the corresponding \textit{Parisi PDE solution} $\Phi=\Phi_{P_1,\Psi,\alpha}$ is defined as follows:
\begin{itemize}
    \item Set $\Phi(1,\cdot) = \phi$ given in~\eqref{e.phi=initial}.
    \item Inductively, for every $l\in \{1,\cdots,K\}$, $s\in [q_{l-1},q_l)$ and $x\in\R^\D$, define
    \begin{align}\label{e.Phi(s,x)}
        \Phi(s,x) = \frac{1}{m_{l-1}}\log\E\exp\left(\m_{l-1} \Phi\left(q_l,\,\sqrt{\mu(q_l)-\mu(s)}g_l+x\right)\right)
    \end{align}
    where $(g_l)_{l\in\{0,\dots,K\}}$ are independent standard $\R^\D$-valued Gaussian vectors.
\end{itemize}
The right-hand side of~\eqref{e.Phi(s,x)} is understood to be $\E \Phi(q_l, \sqrt{\mu(q_l)-\mu(s)}g_l+x)$ if $m_{l-1}=0$.

\end{definition}

By the Cole--Hopf transformation, the next lemma shows that $\Phi_{P_1,\Psi,\alpha}$ indeed satisfies the PDE at continuity points of $\alpha$.
\begin{lemma}\label{l.PDE}
Let $\Phi = \Phi_{P_1,\Psi,\alpha}$ for $\alpha \in \cM_\d$ given in~\eqref{e.alpha=}.
For every $l\in\{1,\dots,K\}$ and every $(s,x)\in [q_{l-1},q_l)\times \R^\D$,
\begin{align*}
    \partial_s \Phi(s,x) + \frac{1}{2}\la \gamma(s)\,,\,\nabla^2\Phi(s,x)+\alpha(s)(\nabla \Phi(s,x))(\nabla \Phi(s,x))^\intercal\ra_{\S^D}=0.
\end{align*}
\end{lemma}

In the display, since $\nabla \Phi$ is $\R^\D$-valued, $(\nabla\Phi)(\nabla\Phi)^\intercal$ is $\S^\D$-valued. 

\begin{proof}
Fix any $l$ and let $s\in [q_{l-1},q_l)$. Set 
\begin{align*}
    w(s,x) = \E\exp\left(\m_{l-1} \Phi\left(q_l,\,\sqrt{\mu(q_l)-\mu(s)}g_l+x\right)\right).
\end{align*}
Using the Gaussian integration by parts (see \cite[Theorem~4.6]{HJbook}) and Lemma~\ref{e.d+d=-gamma}, we can compute
\begin{align*}
    &\partial_s w(s,x)  = \E\Ll[ \m_{l-1}\la \Ll(\frac{\d}{\d s}\sqrt{\mu(q_l)-\mu(s)}\Rr) g_l,\,\nabla \Phi \ra_{\R^\D}e^{\m_{l-1}\Phi}\Rr]
    \\
    &= \E \Ll[ \la -\frac{1}{2}\gamma(s),\  \left(\m_{l-1}^2  (\nabla \Phi )(\nabla \Phi )^\intercal + \m_{l-1}  {\nabla^2\Phi } \right)\ra_{\S^\D}e^{\m_{l-1} \Phi }\Rr],
\end{align*}
where we used the shorthand $\Phi= \Phi\left(q_l,\,\sqrt{\mu(q_l)-\mu(s)}g_l+x\right)$.
We also have
\begin{align*}
    \nabla^2 w(s,x) = \E \Ll[\left(\m_{l-1}^2(\nabla \Phi )(\nabla \Phi )^\intercal+\m_{l-1} \nabla^2\Phi  \right)e^{\m_{l-1}\Phi }\Rr].
\end{align*}
Therefore, we get
\begin{align*}
    \partial_s w(s,x)  + \la \frac{1}{2}\gamma(s), \nabla^2w(s,x)\ra_{\S^D}=0.
\end{align*}
Rewriting $w$ in terms of $\Phi$, we have
\begin{align*}
    \partial_s \Phi + \la \frac{1}{2}\gamma(s),\, \nabla^2 \Phi+ \m_{l-1} (\nabla \Phi)(\nabla \Phi)^\intercal\ra_{\S^D}=0,\quad \text{on $[q_{l-1},q_l)\times\R^\D$}.
\end{align*}
Since $\m_{l-1} = \alpha(s)$ for $s\in[q_{l-1},q_l)$, we have verified the equation. 
\end{proof}

To extend the definition to general $\alpha$, we need the continuity of the functional $\cM_\d \ni\alpha\mapsto \Phi_{P_1,\Psi,\alpha}$. For this, we represent $\Phi_{P_1,\Psi,\alpha}$ in terms of the Ruelle probability cascade (RPC) and use probabilistic methods. We start by introducing the RPC.

\subsection{Ruelle probability cascades}

For $K\in\N$, set
\begin{align*}
    \mathscr{A}=\N^0 \cup \N^1\cup\cdots \cup \N^K
\end{align*}
and we view it as the index set of vertices on a rooted tree, where $\N^0 =\{\emptyset\}$ contains the root and $\N^K$ is the set of leaves. For each leaf $\bal = (n_1,n_2,\dots,n_K) \in \N^K$, we write
\begin{align*}
    p(\bal) = \Ll\{n_1,\, (n_1,n_2),\, \dots,\, (n_1,n_2,\dots,n_K)\Rr\}
\end{align*}
consisting of indices on the shortest path from the root to $\bal$.
For each $\beta\in \mathscr{A}$, we denote by $|\beta|$ the length of $\beta$ satisfying $\beta \in \N^{|\beta|}$.
For $\bal,\bal'\in\mathscr{A}$, we set
\begin{align*}
    \bal\wedge\bal' = |p(\bal)\cap p(\bal')| = \max\{|\beta|:\: \beta\in p(\bal)\cap p(\bal')\}.
\end{align*}
Let $\{\nu_\bal\}_{\bal\in \N^K}$ be the random weights of the RPC given in \cite[(2.46)]{pan} associated with $(m_l)_{l=0}^K$ in~\eqref{e.m,q}. For more detail on its construction, we refer to  \cite[Section~2.3]{pan} or \cite[Section~5.6]{HJbook}.
Instead of $(q_l)_{l=1}^K$ in~\eqref{e.m,q}, we consider a slightly more general sequence
\begin{align}\label{e.q}
    0\leq q_0\leq q_1< \dots< q_K=1.
\end{align}
The only difference is that we allow $q_0\geq 0$.

Fix an infinite-dimensional separable Hilbert space $\mathfrak{H}$. We denote the inner product by a dot. Let $(e_\beta)_{\beta \in \mathscr{A}\setminus \N^0}$ be an orthonormal basis of $\mathfrak{H}$.
For each $\bal\in \N^K$, we set
\begin{align}\label{e.h_alpha=}
    \bh_\bal = \sum_{\beta \in p(\bal)}\sqrt{q_{|\beta|}-q_{|\beta|-1}}e_\beta.
\end{align}
Let $G$ be a random probability measure supported on $(\bh_\bal)_{\bal\in\N^K}$ given by $G(\bh_\bal)= \nu_\bal$. 
Let $(g_\beta)_{\beta\in \mathscr{A}\setminus \N^0}$ be a collection of independent standard $\R^\D$-valued Gaussian vectors.
For each $\bal$, we set
\begin{align}\label{e.Z(h_alpha)=}
    Z(\bh_\bal) = \sum_{\beta \in p(\bal)}\sqrt{\mu(q_{|\beta|})-\mu(q_{|\beta|-1})}g_\beta.
\end{align}
Hence, conditioned on $G$, $(Z(\bh))_{\bh\in \supp G}$ is a centered $\R^\D$-valued Gaussian vector with covariance
\begin{align*}
    \E Z(\bh_\bal)Z(\bh_{\bal'})^\intercal = \mu(q_{\bal\wedge\bal'}) -\mu(q_0)=\mu(\bh_\bal\cdot \bh_{\bal'})-\mu(q_0),
\end{align*}
where we denote the inner product in $\mathfrak{H}$ by a dot.

We also need a continuous version of the Ruelle probability cascade, which will be used to prove the continuity of $\alpha\mapsto \Phi_{P_1,\Psi,\alpha}$ later in Section~\ref{s.Lipschitz} and to define the Parisi functional in Section~\ref{s.pf_main}. 
Let $\mathfrak{R}$ be the RPC on $\mathfrak{H}$ with the overlap $\bh\cdot \bh'$ for $(\bh, \bh')$ sampled from $\E\mathfrak{R}^{\otimes}$ distributed uniformly on $[0,1]$. The existence of $\mathfrak{R}$ is given by \cite[Theorem~2.17]{pan}.
For each $\pi\in\Pi$ in~\eqref{e.Pi}, conditioned on $\mathfrak{R}$, let $(w^{\pi}(\brho))_{\brho\in\supp\mathfrak{R}}$ be a centered $\R^\D$-valued Gaussian process with covariance
\begin{align}\label{e.w^pi}
    \E w^{\pi}(\brho)w^{\pi}(\brho')^\intercal = \pi(\brho\cdot\brho').
\end{align}
The existence and the measurability of this process has been verified in \cite[Section~4]{HJ_critical_pts} (see also \cite[Remark~4.9]{HJ_critical_pts}).

The Parisi functional is usually expressed in terms of the discrete cascade, while sometimes it is easier to work with the continuous one. The following lemma allows us to relate the two.
\begin{lemma}\label{l.RPC_equiv}
For every $\alpha \in \cM_\d$,
\begin{align*}
    \E \log  \iint\exp\Ll(\sigma\cdot Z(\bh)\Rr)\d P_1(\sigma)\d G(\bh)  = \E \log \iint\exp\Ll(\sigma\cdot w^\pi(\brho)\Rr)\d P_1(\sigma)\d \mathfrak{R}(\brho)
\end{align*}
where $\pi(s) = \mu\circ\zeta(s) - \mu\circ\zeta(0)$ for $s\in[0,1]$.
\end{lemma}

In the above, $\E$ on the left (resp.\ right) integrates the Gaussian randomness in $Z(\bh)$ (resp.\ $w^\pi(\brho)$) and the randomness in $G$ (resp.\ $\mathfrak{R}$).

\begin{proof}
We denote the left- and right-hand sides by $\mathrm{LHS}$ and $\mathrm{RHS}$, respectively.
Let $(\bg^l)_{l\in\N}$ and $(\brho^l)_{l\in\N}$ be sampled from $G^{\otimes\infty}$ and $\mathfrak{R}^{\otimes\infty}$, respectively. We write $R^{l,l'} = \bg^l\cdot \bg^{l'}$ and $Q^{l,l'} = \zeta(\brho^l\cdot\brho^{l'})$. 
Notice
\begin{align*}
    \E Z\Ll(\bg^l\Rr)Z\Ll(\bg^{l'}\Rr)^\intercal = \mu\Ll(R^{l,l'}\Rr)-c,\qquad \E w^\pi\Ll(\brho^l\Rr)w^\pi\Ll(\brho^{l'}\Rr)^\intercal = \mu\Ll(Q^{l,l'}\Rr)-c
\end{align*}
where $c=\mu(q_0)=\mu\circ\zeta(0)$.
Hence, the two sides in the announced identity are of the same structure. By the straightforward adaptation of \cite[Theorem~1.3]{pan}, for every $\eps>0$, there are some $n\in\N$ and a bounded continuous function $F_\eps:\R^{n\times n}\to \R$ such that
\begin{align*}
    \Ll|\mathrm{RHS}-\E \int F_\eps \Ll(R^n\Rr)\d G^{\otimes n} \Rr| \leq \eps,\qquad
    \Ll|\mathrm{LHS}-\E \int F_\eps \Ll(Q^n\Rr)\d \mathfrak{R}^{\otimes n} \Rr| \leq \eps
\end{align*}
where $R^n = \Ll(R^{l,l'}\Rr)_{l,l'\leq n}$ and $Q^n = \Ll(Q^{l,l'}\Rr)_{l,l'\leq n}$.

Then, we want to show that $R^n$ and $Q^n$ have the same distribution. For any continuous bounded function $g:\R\to\R$, by the standard result on the discrete cascade \cite[(2.81)]{pan}, we have
\begin{align*}
    \E G^{\otimes 2}\Ll(g\Ll(R^{1,2}\Rr)\Rr) = \sum_{l=0}^Kg(q_l)(m_l-m_{l-1}) = \int g(s)\d \alpha(s).
\end{align*}  
By the definition of $\mathfrak{R}$, we have
\begin{align*}
    \E \mathfrak{R}^{\otimes 2}\Ll(g\Ll(Q^{1,2}\Rr)\Rr) = \E \mathfrak{R}^{\otimes 2}\Ll(g\Ll(\zeta\Ll(\brho^1\cdot\brho^2\Rr)\Rr)\Rr)=\int_0^1  g(\zeta(s))\d s = \int g(s)\d \alpha(s).
\end{align*}
Hence, the overlap $R^{1,2}$ under $\E G^{\otimes 2}$ has the same distribution of $Q^{1,2}$ under $\E \mathfrak{R}^{\otimes 2}$. Moreover, by construction and the property of the RPC, both $R=\Ll(R^{l,l'}\Rr)_{l,l'\in\N}$ and $Q=\Ll(Q^{l,l'}\Rr)_{l,l'\in\N}$ satisfy the Ghirlanda--Guerra identities \cite[(2.80)]{pan}. Since symmetric random arrays satisfying the Ghirlanda--Guerra identities are uniquely characterized by the distribution of the overlap \cite[Theorem~2.13]{pan}, $R$ under $\E G^{\otimes \infty}$ and $Q$ under $\E \mathfrak{R}^{\otimes \infty}$ have the same distribution. In particular,
\begin{align*}
    \E \int F_\eps\Ll(R^n\Rr)\d G^{\otimes n} = \E \int F_\eps\Ll(Q^n\Rr)\d \mathfrak{R}^{\otimes n}
\end{align*}
and thus $|\mathrm{RHS}- \mathrm{LHS}|\leq 2\eps$. Since $\eps$ is arbitrary, the desired result follows.
\end{proof}
\subsection{Cascade computation}

We review the standard cascade computation and then obtain a representation of $\Phi_{P_1,\Psi,\alpha}$ in terms of $\mathfrak{R}$. 
Define
\begin{align*}
    X_K = \phi\Ll(\sum_{l=1}^K\sqrt{\mu(q_l)-\mu(q_{l-1})}g_l\Rr).
\end{align*}
Inductively, for $l\in \{1,\dots,K\}$, define
\begin{align*}
    X_{l-1} = \frac{1}{m_{l-1}} \log \E_{g_l}\exp\Ll(m_{l-1}X_l\Rr),
\end{align*}
where $\E_{g_l}$ integrates $g_l$.
The next lemma recalls the standard computation.

\begin{lemma}\label{l.X_0=}
Let $X_0$ be defined as above, which is associated with $(m_l)_{l=0}^K$ in \eqref{e.m,q} and $(q_l)_{l=0}^K$ in \eqref{e.q}. Then,
\begin{align*}
    X_0 = \E \log \iint \exp \Ll(\sigma\cdot Z(\bh) \Rr)\d P_1(\sigma) \d G(\bh).
\end{align*}
\end{lemma}
\begin{proof}
By \cite[Theorem~2.9]{pan},
\begin{align*}
    X_0  = \E \log \sum_{\bal\in \N^K}\nu_\bal \int \exp\Ll(\sigma\cdot \sum_{\beta\in p(\bal)}\sqrt{\mu(q_{|\beta|}) -\mu(q_{|\beta|-1})}g_\beta\Rr)\d P_1(\sigma).
\end{align*}
Using~\eqref{e.h_alpha=},~\eqref{e.Z(h_alpha)=}, and the lines in between, we can rewrite the above into the desired form.
\end{proof}

\begin{lemma}\label{l.Phi(s,x)=RPC}
For $\alpha \in \cM_\d$ and $s\in [0,1)$, let $\bar\alpha = \alpha\mathds{1}_{[s,\infty)}$. Then,
\begin{align*}
    \Phi_{P_1,\Psi,\bar\alpha}(s,x) = \E \log \iint \exp\Ll(\sigma\cdot w^{\pi}(\brho)+\sigma\cdot x\Rr)\d P_1(\sigma)\d \mathfrak{R}(\brho),\quad\forall x\in \R^\D,
\end{align*}
where $\pi(r) = \mu\circ \bar\alpha^{-1}(r) - \mu(s)$ for all $r\in[0,1]$.
\end{lemma}

\begin{proof}
Fix $l\in \{0,\dots,K\}$ such that $s\in [q_{l-1},q_l)$. We set $\bar K = K-l+1$,
\begin{align}\label{e.bar_m,bar_q}
    \bar q_k=
    \begin{cases}
        q_{k+l-1},\quad & k\in \{1,\dots,\bar K\},
        \\
        s, \quad & k=0;
    \end{cases}
    \qquad
    \bar m_k =
    \begin{cases}
        m_{k+l-1},\quad & k\in\{0,1,\dots,\bar K\},
        \\
        0,\quad & k=-1.
    \end{cases}
\end{align}
Then, we have $\bar \alpha = \sum_{k=0}^{\bar K} (\bar m_{k}-\bar m_{k-1})\mathds{1}_{[\bar q_k,\infty)}$. 
We define
\begin{align*}
    \bar X_{\bar K} = \phi \Ll(x+ \sum_{k=1}^{\bar K}\sqrt{\mu(\bar q_k)-\mu(\bar q_{k-1})}g_k\Rr)
\end{align*}
and iteratively
\begin{align*}
    \bar X_{k-1} = \frac{1}{\bar m_{k-1}}\log\E_{g_k}\exp\Ll(\bar m_{k-1} \bar X_k\Rr).
\end{align*}
Comparing it with \eqref{e.Phi(s,x)}, we can inductively verify
\begin{align*}
    \bar X_k = \Phi_{P_1,\Psi,\bar\alpha}\Ll(\bar q_k, x+\sum^k_{k'=0}\sqrt{\mu(\bar q_{k'})-\mu(\bar q_{k'-1})}g_{k'}\Rr),\quad\forall k\in\{0,\dots,\bar K\}.
\end{align*}
In particular, $\bar X_0 = \Phi_{P_1,\Psi,\bar\alpha}(s,x)$. Also notice $\bar\alpha^{-1}(0) =s$ due to~\eqref{e.alpha^-1=}. Hence, applying Lemma~\ref{l.X_0=} and Lemma~\ref{l.RPC_equiv} (with $e^{\sigma\cdot x} \d P_1(\sigma)$ substituted for $P_1$) to $\bar\alpha$ and sequences in \eqref{e.bar_m,bar_q}, we can obtain the announced result.
\end{proof}

\begin{corollary}\label{c.EPhi=}
Let $\eta$ be a standard $\R^\D$-valued Gaussian vector.
For every $\alpha\in\cM_\d$,
\begin{align*}
    \E\Ll[\Phi_{P_1,\Psi,\alpha}\Ll(0,\sqrt{\mu(0)}\eta\Rr)\Rr]=\E\log\iint \exp\Ll(\sigma\cdot w^{\mu\circ\zeta}(\brho)\Rr)\d P_1(\sigma)\d\mathfrak{R}(\brho).
\end{align*}
\end{corollary}
\begin{proof}
Assuming that $\eta$ is independent from all other randomness, by~\eqref{e.w^pi}, we have $w^{\pi}(\brho)+ \sqrt{\mu(0)}\eta \stackrel{\d}{=} w^{\mu\circ\zeta}(\bh)$
for $\pi = \mu\circ\zeta- \mu(0)$. Using this and Lemma~\ref{l.Phi(s,x)=RPC} with $s=0$ (thus $\bar\alpha=\alpha$), we can get the above expression.
\end{proof}

\subsection{Lipschitz continuity}\label{s.Lipschitz}

After the next lemma, we will show that $\cM_\d \ni\alpha \mapsto \Phi_{P_1,\Psi,\alpha}$ is Lipschitz.
\begin{lemma}\label{l.Lipschitz_pi}
For every $\pi,\pi'\in\Pi$ with $\pi(1)= \pi'(1)$,
\begin{align*}
    \Ll|\E \log \iint \exp\Ll(\sigma\cdot w^{\pi}(\brho)\Rr)\d P_1(\sigma)\d \mathfrak{R}(\brho)- \E \log \iint \exp\Ll(\sigma\cdot w^{\pi'}(\brho)\Rr)\d P_1(\sigma)\d \mathfrak{R}(\brho)\Rr|
    \\
    \leq \frac{1}{2}\int_0^1 \Ll|\pi(s)-\pi'(s)\Rr|\d s.
\end{align*}
\end{lemma}
\begin{proof}
We can assume that $w^\pi(\brho)$ is independent of $w^{\pi'}(\brho)$. For every $r\in[0,1]$, we set
\begin{align*}
    \varphi(r)=\E \log \iint \exp\Ll(\sigma\cdot \Ll(\sqrt{1-r}w^\pi(\brho) + \sqrt{r}w^{\pi'}(\brho)\Rr)\Rr)\d P_1(\sigma)\d \mathfrak{R}(\brho).
\end{align*}
Using the Gaussian integration by parts (see \cite[Theorem~4.6]{HJbook}) and \eqref{e.w^pi}, we can compute
\begin{align*}
    \frac{\d}{\d r}\varphi(r) & = \frac{1}{2}\E \la \sigma\cdot \Ll(\frac{1}{\sqrt{r}}w^{\pi'}(\brho) -\frac{1}{\sqrt{1-r}}w^\pi(\brho)  \Rr)\ra_r
    \\
    & = \frac{1}{2}\E \la \sigma\sigma^\intercal \cdot \Ll(\pi(1)-\pi'(1)\Rr) - \sigma\sigma'^\intercal\cdot \Ll(\pi(\brho\cdot\brho')-\pi'(\brho\cdot\brho')\Rr)\ra_r 
\end{align*}
where $\la\cdot\ra_r$ is the obvious Gibbs measure and $(\sigma',\brho')$ is an independent copy of $(\sigma,\brho)$ under $\la\cdot\ra_r$. 
The assumption \eqref{e.supp_P_1} on $P_1$ implies $|\sigma\sigma'^\intercal|\leq 1$.
The invariance of the RPC \cite[Theorem~4.4]{pan} (and \cite[Proposition~4.8]{HJ_critical_pts}) implies that the law of $\brho\cdot\brho'$ under $\E \la\cdot\ra_r$ is still uniform over $[0,1]$. Using these and $\pi(1)=\pi'(1)$, we get
\begin{align*}
    \Ll|\frac{\d}{\d r}\varphi(r)\Rr|\leq \frac{1}{2}\E \la \Ll|\pi(\brho\cdot\brho')-\pi'(\brho\cdot\brho')\Rr|\ra_r = \frac{1}{2}\int_0^1 \Ll|\pi(s)-\pi'(s)\Rr|\d s
\end{align*}
which yields the announced bound.
\end{proof}

For the SK model, the following Lipschitz result is attributed to \cite{gue03} (see also \cite[Proposition~3.1]{talagrand2006parisi}).
\begin{lemma}\label{l.Lipschitz_phi}
There is a constant $C>0$ depending only on $\xi$ and $\Psi$ such that
\begin{align*}
    \sup_{[0,1]\times \R^\D } \Ll|\Phi_{P_1,\Psi,\alpha}(s,x) - \Phi_{P_1,\Psi,\alpha'}(s,x)\Rr|\leq C\int_0^1\Ll|\alpha(s)-\alpha'(s)\Rr|\d s, \quad\forall \alpha,\alpha'\in \cM_\d.
\end{align*}
\end{lemma}

\begin{proof}
Fix any $(s,x)$. If $s=1$, then there is nothing to show. Hence, we assume $s\in[0,1)$. As in the setting of Lemma~\ref{l.Phi(s,x)=RPC}, let $\bar\alpha=\alpha\mathds{1}_{[s,\infty)}$ and $\bar\alpha'=\alpha'\mathds{1}_{[s,\infty)}$. 
Then, we set $\pi=\mu\circ\bar\alpha^{-1}-\mu(s)$ and similarly for $\pi'$.
Due to $\alpha,\alpha'\in \cM_\d$ and~\eqref{e.alpha^-1=}, we have $\bar\alpha^{-1}(1)=\bar\alpha'^{-1}(1)=1$ and thus $\pi(1) =\pi'(1)$. Hence, we can apply Lemma~\ref{l.Phi(s,x)=RPC} and Lemma~\ref{l.Lipschitz_pi} (with $e^{\sigma\cdot x}\d P_1(\sigma)$ substituted for $P_1$) to get
\begin{align*}
    \Ll|\Phi_{P_1,\Psi,\alpha}(s,x) - \Phi_{P_1,\Psi,\alpha'}(s,x)\Rr|\leq \frac{1}{2}\int_0^1 \Ll|\pi(s)-\pi'(s)\Rr|\d s \leq C\int_0^1 \Ll|\bar\alpha^{-1}(s)-\bar\alpha'^{-1}(s)\Rr|\d s
\end{align*}
where $C$ arises from the local Lipschitzness of $\mu$ (as a result of the local Lipschitzness of $\nabla\xi$ in \ref{i.xi_loc_lip} and $\Psi\in \Pi_\mathrm{Lip}(z)$). It is well known (as a result of Fubini's theorem) that 
\begin{align*}
    \int_0^1 \Ll|\bar\alpha^{-1}(s)-\bar\alpha'^{-1}(s)\Rr|\d s = \int_0^1 \Ll|\bar\alpha(s)-\bar\alpha'(s)\Rr|\d s
\end{align*}
which yields the desired bound.
\end{proof}

Recall the metric on $\cM$ in \eqref{e.d_M=}. Therefore, $\alpha\mapsto \Phi_{P_1,\Psi,\alpha}$ is Lipschitz from $\cM_\d$ to $C([0,1]\times \R^\D)$ equipped with the uniform norm. Hence, we can extend Definition~\ref{d.sol_pde_discrete} to general $\alpha\in\cM$.

\begin{definition}[Parisi PDE solution]\label{d.sol_pde}

Given $P_1$ satisfying~\eqref{e.supp_P_1}, $\Psi\in\Pi_\mathrm{Lip}(z)$ satisfying~\eqref{e.weak_gamma>0}, and $\alpha\in\cM$, the associated \textit{Parisi PDE solution} $\Phi_{P_1,\Psi,\alpha}$ is the limit in $C([0,1]\times \R^\D)$ 
of $\Ll(\Phi_{P_1,\Psi,\alpha_n}\Rr)_{n\in\N}$ (given in Definition~\ref{d.sol_pde_discrete}) for any sequence $(\alpha_n)_{n\in\N}$ in $\cM_\d$ converging to $\alpha$ in $\cM$.
\end{definition}

We state an immediate implication of Lemma~\ref{l.Lipschitz_phi}.
\begin{corollary}\label{c.well-defined}
Let $P_1$ satisfy \eqref{e.supp_P_1} and $\Psi\in \Pi_\mathrm{Lip}(z)$ satisfy \eqref{e.weak_gamma>0}. For every $\alpha\in\cM$, $\Phi_{P_1,\Psi,\alpha}$ given in Definition~\ref{d.sol_pde} is well-defined. Namely, for any sequence $(\alpha_n)_{n\in\N}$ in $\cM_\d$ converging to $\alpha$ in $\cM$, the limit of $\Ll(\Phi_{P_1,\Psi,\alpha_n}\Rr)_{n\in\N}$ (given in Definition~\ref{d.sol_pde_discrete}) in $C([0,1]\times \R^\D)$ exists and does not depend on the choice of $(\alpha_n)_{n\in\N}$. As a result, if $\alpha \in \cM_\d$, then $\Phi_{P_1,\Psi,\alpha}$ is equal to the one given in Definition~\ref{d.sol_pde_discrete}.

Moreover, the terminal condition $\Phi_{P_1,\Psi,\alpha}(1,\cdot) = \phi$ given in~\eqref{e.phi=initial} is satisfied for every $\alpha\in\cM$; and there is a constant $C>0$ depending only on $\xi$ and $\Psi$ such that
\begin{align*}
    \sup_{[0,1]\times \R^\D } \Ll|\Phi_{P_1,\Psi,\alpha}(s,x) - \Phi_{P_1,\Psi,\alpha'}(s,x)\Rr|\leq C\int_0^1\Ll|\alpha(s)-\alpha'(s)\Rr|\d s,\quad\forall \alpha,\alpha'\in \cM.
\end{align*}
\end{corollary}

If $\alpha\in\cM$ is continuous, we will show in Proposition~\ref{p.reg_Phi}~\eqref{i.dsdxPhi_exists} that $\Phi_{P_1,\Psi,\alpha}(1,\cdot)$ satisfies the Parisi PDE~\eqref{e.Parisi_PDE} in the classical sense.

\subsection{Regularity of the solution}
Let $\mathscr{I} = \cup_{k\in\{0\}\cup \N} \{1,\dots,D\}^k$ be the set of multi-indices.
For $\bi = (i_1,i_2,\dots,i_k)\in \{1,\dots,D\}^k$ for some $k\in\N$, we set $|\bi|=k$ and $\partial^\bi_x = \partial_{x_{i_1}}\partial_{x_{i_2}}\cdots \partial_{x_{i_k}}$. We set $\partial^\emptyset$ to be identity operator. For $k\in\N$, we write $\nabla^k = \Ll(\partial^\bi_x\Rr)_{\bi:\:|\bi|=k}$ and $\nabla^{\leq k} = \Ll(\partial^\bi_x\Rr)_{\bi:\:1\leq |\bi|\leq k}$.

For $r, r'\in\R$, we write 
$r\vee r' = \max\{r,r'\}$.
Let $(B(t))_{t\geq 0}$ be the standard $\R^\D$-valued Wiener process (i.e.\ $D$-dimensional Brownian motion). 
We consider the $\R^\D$-valued martingale
\begin{align}\label{e.M(s)=}
    M(s) = \int_0^s \sqrt{\gamma(t)}\d B(t) = \Ll(\int_0^s \sum_{j=1}^\D\Ll(\sqrt{\gamma(t)}\Rr)_{kj} \d B_j(t)\Rr)_{1\leq k\leq \D},\quad\forall s\in[0,1].
\end{align}
For $0\leq s\leq t\leq 1$, we have
\begin{align}\label{e.M(t)-M(s)}
    M(t) - M(s) \stackrel{\d}{=} \sqrt{\mu(t)-\mu(s)}\eta
\end{align}
where $\eta$ is the standard $\R^\D$-valued Gaussian vector.
Throughout this subsection, we write $\Phi_\alpha = \Phi_{P_1,\Psi,\alpha}$.
For $(s,x) \in [0,1]\times\R^\D$, we set
\begin{gather}
V_\alpha(s,x) =  \Phi_\alpha(1,x+M(1)-M(s)) -\int_0^1 \Phi_\alpha(s\vee t,x+M(s\vee t)-M(s)) \d \alpha(t).\label{e.V=}
\end{gather}
We emphasize that $V$ depends on $\alpha$.

First, we verify the regularity of the Parisi PDE solution for discrete $\alpha$.
\begin{lemma}\label{l.deri_disc}
For every $\bi\in \mathscr{I}$, there are constants $C_{1,\bi}$ and $C_{0,\bi}$ and, if $|\bi|\geq 1$, a polynomial $F_\bi:\R^{\sum_{k=1}^{|\bi|-1}D^k}\times\R^{\sum_{k=1}^{|\bi|}D^k}\to \R$ such that the following holds for every $\alpha\in\cM_\d$:
\begin{gather}
    \partial^\bi_x \Phi_\alpha \in C([0,1]\times \R^\D),\label{e.d^iPhi_in_C}
    \\
    \sup_{l\in\{1,\dots,K\}} \sup_{[q_{l-1},q_l)\times \R^\D} \Ll|\partial_s \partial^\bi_x \Phi_\alpha\Rr| \leq C_{1,\bi}, \label{e.|d_sd_xPhi|}
\end{gather}
and, if $|\bi|\geq 1$,
\begin{gather}
    \partial^\bi_x \Phi_\alpha(s,x) = \E \Ll[F_\bi\Ll(\nabla^{\leq|\bi|-1}V_\alpha(s,x),\,\nabla^{\leq|\bi|}\phi(x+M(1)-M(s))\Rr)e^{V_\alpha(s,x)}\Rr],\quad\forall (s,x),\label{e.d^iPhi=E}
    \\
    \sup_{[0,1]\times \R^\D} \Ll|\partial^\bi_x\Phi_\alpha\Rr|\leq C_{0,\bi}. \label{e.|d^iPhi|<C}
\end{gather}

\end{lemma}

This is basically \cite[Proposition~5]{auffinger2015properties}.
The difference is that there $\D =1$ and $\phi =\log\cosh$ and specific properties of $\log\cosh$ were used. In fact, one only needs the following property of $\phi$.

\begin{lemma}\label{l.|dphi|<infty}
For every $\bi\in\mathscr{I}$ with $|\bi|\geq 1$, $\sup_{\R^\D} |\partial^\bi_x\phi|<\infty$.
\end{lemma}
\begin{proof}
For each $x\in \S^\D$, we consider the (nonrandom) Gibbs measure
\begin{align}\label{e.<>_x}
    \la\cdot\ra_x\propto e^{x\cdot \sigma}\d P_1(\sigma).
\end{align}
We can inductively compute
\begin{align*}
    \partial^\bi_x\phi = \la G_\bi\Ll(\Ll(\sigma^l_k\Rr)_{1\leq k\leq D,\, 1\leq l\leq |\bi|}\Rr) \ra_x 
\end{align*}
where $G_\bi:\R^{D\times |\bi|}\to\R$ is a polynomial and $(\sigma^l)_{1\leq l\leq |\bi|}$ are independent copies of $\sigma$ under $\la\cdot\ra_x$. The assumption \eqref{e.supp_P_1} on $P_1$ ensures that $|\sigma^l|\leq 1$ for every $l$ and thus the lemma follows.
\end{proof}

To prove Lemma~\ref{l.deri_disc}, we need one more lemma.

\begin{lemma}\label{l.EeV=1}
If $\alpha\in\cM_\d$, then $\E e^{V_\alpha(s,x)}=1$ for every $(s,x) \in [0,1]\times \R^\D$.
\end{lemma}
\begin{proof}
Write $\Phi=\Phi_\alpha$ and $V=V_\alpha$.
Assuming $\alpha$ has the form in~\eqref{e.alpha=}, we have $\d \alpha =\sum_{l=0}^K (m_l-m_{l-1})\boldsymbol\delta_{q_l}$. Then, for $s \in [q_{l-1},q_l)$, the definition of $V$ in~\eqref{e.V=} yields
\begin{align*}
    &V(s,x)  = \Phi(1,x+M(1)-M(s)) - \sum_{k=0}^K (m_k-m_{k-1}) \Phi(s\vee q_k,x+M(s\vee q_k)-M(s))
    \\
    &=  \Phi(1,x+M(1)-M(s)) - \sum_{k=l}^K (m_k-m_{k-1}) \Phi( q_k,x+M( q_k)-M(s)) - m_{l-1} \Phi(s,x).
\end{align*}
Directly applying this, we have
\begin{align*}
    &V(q_l,x+M(q_l)-M(s))=  \Phi(1,x+M(1)-M(s)) 
    \\
    &- \sum_{k=l+1}^K (m_k-m_{k-1}) \Phi( q_k,x+M( q_k)-M(s)) - m_{l} \Phi(q_l,x+M(q_l)-M(s)).
\end{align*}
Comparing the above two displays, we get, for $s\in [q_{l-1},q_l)$, 
\begin{align}\label{e.V+..=V}
    V(s,x)=V(q_l,x+M(q_l)-M(s))+m_{l-1} \Ll(\Phi(q_l,x+M(q_l)-M(s))-\Phi(s,x)\Rr) .
\end{align}

To prove the lemma, we use induction. Due to $V(1,x)=0$, we have $\E e^{V(1,x)}=1$ for all $x\in\R^\D$. Now, we assume $\E e^{V(s,x)}=1$ for all $(s,x)\in [q_l,1]\times \R^\D$. Let $s\in [q_{l-1},q_l)$. Using~\eqref{e.V+..=V}, we have
\begin{align*}
    \E e^{V(s,x)} = \E \Ll[e^{m_{l-1} \Ll(\Phi(q_l,x+M(q_l)-M(s))-\Phi(s,x)\Rr)}\E\Ll[ e^{V(q_l,x+M(q_l)-M(s))}\,\big|\,M(q_l)-M(s)\Rr]\Rr].
\end{align*}
Since $V(q_l,x)$ only depends on $(M(q_k)-M(q_l))_{k>l}$ which are independent of $M(q_l)-M(s)$, the induction assumption implies that the conditional expectation is equal to $1$. Due to~\eqref{e.Phi(s,x)} and~\eqref{e.M(t)-M(s)}, we get
\begin{align*}
    \E \Ll[e^{m_{l-1} \Ll(\Phi(q_l,x+M(q_l)-M(s))-\Phi(s,x)\Rr)}\Rr]=1
\end{align*}
which completes the induction step and the proof of the lemma.    
\end{proof}

\begin{proof}[Proof of Lemma~\ref{l.deri_disc}]
Note that~\eqref{e.d^iPhi_in_C} follows from~\eqref{e.|d^iPhi|<C} and~\eqref{e.|d_sd_xPhi|}  from~\eqref{e.|d^iPhi|<C} and Lemma~\ref{l.PDE}. Hence, we only need to show~\eqref{e.d^iPhi=E} and~\eqref{e.|d^iPhi|<C}.

Write $\Phi=\Phi_\alpha$ and $V=V_\alpha$.
We first show~\eqref{e.d^iPhi=E} and start with $\bi$ satisfy $|\bi|=1$. Assuming that $\bi=(i)\in\{1,\dots,\D\}$, We want to verify, for all $(s,x) \in [0,1]\times \R^\D$,
\begin{align}\label{e.partialPhi_i=1}
    \partial_{x_i}\Phi(s,x) = \E \partial_{x_i} \phi(x+M(1)-M(s)) e^{V(s,x)}
\end{align}
To show this, we use backwards induction on the range of $s$. Due to $V(1,x)=0$, we have
\begin{align*}
    \partial_{x_i}\Phi(1,x) = \partial_{x_i}\phi(x) = \E \partial_{x_i}\phi(x+M(1)-M(1))e^{V(1,x)}.
\end{align*}
Now, we assume that~\eqref{e.partialPhi_i=1} holds for all $(s,x) \in [q_l,1]\times \R^\D$. By~\eqref{e.Phi(s,x)} (see also~\eqref{e.M(t)-M(s)}), \eqref{e.partialPhi_i=1} at $s=q_l$, and~\eqref{e.V+..=V}, we can compute
\begin{align*}
    \partial_{x_i} \Phi(s,x)  &= \frac{\E \partial_{x_i} \Phi(q_l,x+M(q_l)-M(s))e^{m_{l-1}\Phi(q_l,x+M(q_l)-M(s))}}{\E e^{m_{l-1}\Phi(q_l,x+M(q_l)-M(s))}}
    \\
    & = \E \partial_{x_i} \Phi(q_l,x+M(q_l)-M(s))e^{m_{l-1}(\Phi(q_l,x+M(q_l)-M(s))-\Phi(s,x))}
    \\
    & = \E \partial_{x_i}\phi(x+M(1)-M(s))e^{V(q_l,x+M(q_l)-M(s))+m_{l-1}(\Phi(q_l,x+M(q_l)-M(s))-\Phi(s,x))}
    \\
    & = \E \partial_{x_i}\phi(x+M(1)-M(s))e^{V(s,x)}.
\end{align*}
This completes the induction step and verifies~\eqref{e.partialPhi_i=1}. Therefore, we have shown~\eqref{e.d^iPhi=E} for $\bi$ satisfying $|\bi|=1$. For $|\bi|>1$,~\eqref{e.d^iPhi=E} follows from induction and basic calculus rules. 

Now, we turn to~\eqref{e.|d^iPhi|<C}. For $\bi$ satisfying $|\bi|=1$,~\eqref{e.|d^iPhi|<C} follows from~\eqref{e.partialPhi_i=1}, Lemma~\ref{l.|dphi|<infty}, and Lemma~\ref{l.EeV=1}. For $\bi$ satisfying $|\bi|>1$, we use induction and the expression of $V$ to bound $\nabla^{\leq |\bi|-1 }V(s,x)$ and use Lemma~\ref{l.|dphi|<infty} to bound $\nabla^{\leq |\bi|}\phi(x+M(1)-M(s))$. Then,~\eqref{e.|d^iPhi|<C} follows from these bounds,~\eqref{e.d^iPhi=E}, and Lemma~\ref{l.EeV=1}.
\end{proof}

The next lemma is a straightforward adaptation of \cite[Lemma~6]{auffinger2015properties}.

\begin{lemma}[\cite{auffinger2015properties}]\label{l.P(V)_cvg}
Suppose that a sequence $(\alpha_n)_{n\in\N}$ in $\cM_\d$ converges to $\alpha$ in $\cM$. For $k\geq 1$, let $P:\R^{\sum_{i=1}^kD^i}\to\R$ be a polynomial. If $\nabla^{\leq k}\Phi_\alpha$ exists and $ \nabla^{\leq k}\Phi_{\alpha_n}$ converges to $\nabla^{\leq k}\Phi_\alpha$ uniformly on $[0,1]\times \R^\D$, then
\begin{align*}
    \lim_{n\to\infty}\sup_{[0,1]\times\R^\D}\E \Ll|P\Ll(\nabla^{\leq k}V_{\alpha_n}\Rr)e^{V_{\alpha_n}}-P\Ll(\nabla^{\leq k}V_\alpha\Rr)e^{V_\alpha}\Rr|=0.
\end{align*}
\end{lemma}

Heuristically, this is evident from the definition of $V_\alpha$ in~\eqref{e.V=} that $\nabla^{\leq k}V_\alpha$ is an integral of $\nabla^{\leq k}\Phi_\alpha$. We refer to \cite{auffinger2015properties} for the detailed proof.

\begin{proposition}\label{p.reg_Phi}
The following holds:
\begin{enumerate}
    \item \label{i.d_xPhi_exists} For every $\alpha\in\cM$ and every $\bi\in\mathscr{I}$, the real-valued function $\partial^\bi_x\Phi_\alpha$ exists and is continuous on $[0,1]\times \R^\D$.
    \item \label{i.dxPhi_cvg} If $(\alpha_n)_{n\in\N}$ converges to some $\alpha $ in $\cM$, then, for every $\bi\in\mathscr{I}$, $\partial^\bi_x\Phi_{\alpha_n}$ converges to $\partial^\bi_x\Phi_\alpha$ uniformly on $[0,1]\times \R^\D$.
    \item \label{i.sup|d_xPhi|<infty} For every $\bi\in\mathscr{I}$ with $|\bi|\geq 1$, 
    \begin{align*}
        \sup_{\alpha\in\cM}\sup_{[0,1]\times \R^\D}\Ll|\partial^\bi_x\Phi_\alpha\Rr|<\infty. \end{align*}
    \item \label{i.dsdxPhi_exists} If $\alpha \in \cM$ is continuous, then for every $\bi\in\mathscr{I}$, the real-valued function $\partial_s\partial^\bi_x\Phi_\alpha$ exists and is continuous on $[0,1]\times \R^\D$.
    In particular, for every $(s,x)\in[0,1]\times\R^\D$,
    \begin{align}\label{e.PDE_cts}
    \partial_s \Phi_\alpha(s,x)  +\frac{1}{2}\la \gamma(s)\,,\,\nabla^2\Phi_\alpha(s,x)+\alpha(s)(\nabla \Phi_\alpha(s,x))(\nabla \Phi_\alpha(s,x))^\intercal\ra_{\S^D}=0.
\end{align}
\end{enumerate}
\end{proposition}
\begin{proof}
Part~\eqref{i.d_xPhi_exists}. Let $(\alpha_n)_{n\in\N}$ be a sequence in $\cM_\d$ that converges to $\alpha$. If $|\bi|=0$, then the uniform convergence of $\Phi_{\alpha_n}$ to $\Phi_\alpha$ given by Lemma~\ref{l.Lipschitz_phi} implies that $\Phi_\alpha$ is continuous. Now, we proceed by induction and assume that part~\eqref{i.d_xPhi_exists} holds for all $\bi'$ with $|\bi'|\leq k-1$ for some $k\in\N$. Let $\bi$ satisfy $|\bi|=k$. Let $F_\bi$ be given in Lemma~\ref{l.deri_disc}. We rewrite it as 
\begin{align*}
    F_\bi(\mathbf{x},\mathbf{y}) = \sum_{l=1}^pP_l(\mathbf{x})Q_l(\mathbf{y}), \quad\forall (\mathbf{x},\mathbf{y})\in \R^{\sum_{k=1}^{|\bi|-1}D^k}\times\R^{\sum_{k=1}^{|\bi|}D^k}
\end{align*}
for some $p\in\N$,
where $P_l$ and $Q_l$ are polynomials. The induction assumption together with Lemma~\ref{l.|dphi|<infty} and Lemma~\ref{l.P(V)_cvg} implies that, for every $l$,
\begin{align*}
    \lim_{n\to\infty} \sup_{(s,x)\in[0,1]\times \R^\D}\E\Big[\Ll|P_l\Ll(\nabla^{\leq |\bi|-1}V_{\alpha_n}(s,x)\Rr)e^{V_{\alpha_n}(s,x)}-P_l\Ll(\nabla^{\leq |\bi|-1}V_\alpha(s,x)\Rr)e^{V_\alpha(s,x)}\Rr|
    \\
    \times\Ll|Q_l\Ll(\nabla^{\leq |\bi|}\phi(x+M(1)-M(s))\Rr)\Rr|\Big]=0.
\end{align*}
This along with~\eqref{e.d^iPhi=E} implies that $\partial^\bi_x\Phi_{\alpha_n}$ converges uniformly on $[0,1]\times \R^\D$. 
Also, $\partial^\bi_x\Phi_{\alpha_n}$ is continuous by~\eqref{e.d^iPhi_in_C} in Lemma~\ref{l.deri_disc}.
So, $\partial^\bi_x\Phi_\alpha$ exists and is continuous. 
This completes the induction step and the proof of part~\eqref{i.d_xPhi_exists}.

Part~\eqref{i.dxPhi_cvg}. In the proof of part~\eqref{i.d_xPhi_exists}, we have shown the convergence for discrete $(\alpha_n)_{n\in\N}$. Since $\cM_\d$ is dense in $\cM$, we can deduce this part by approximation.

Part~\eqref{i.sup|d_xPhi|<infty} follows from part~\eqref{i.dxPhi_cvg} and~\eqref{e.|d^iPhi|<C}.

Part~\eqref{i.dsdxPhi_exists}. We approximate $\alpha$ by $(\alpha_n)_{n\in\N}$ in $\cM_\d$. Integrating the equation for each $\alpha_n$ in Lemma~\ref{l.PDE}  and using part~\eqref{i.dxPhi_cvg}, we get, for every $(s,x)\in[0,1]\times \R^\D$,
\begin{align*}
    \phi(x)-\Phi_\alpha(s,x)=-\int_s^1\frac{1}{2}\la \gamma(t)\,,\,\nabla^2\Phi_\alpha(t,x)+\alpha(t)(\nabla \Phi_\alpha(t,x))(\nabla \Phi_\alpha(t,x))^\intercal\ra_{\S^D}\d t.
\end{align*}
By part~\eqref{i.d_xPhi_exists}, we conclude that $\partial_s\Phi_\alpha$ exists and is continuous and also~\eqref{e.PDE_cts} holds. This proves part~\eqref{i.dsdxPhi_exists} for $\bi$ satisfying $|\bi|=0$. The general case follows from differentiating the equation for $\alpha_n$ in Lemma~\ref{l.PDE} further (allowed by~\eqref{e.d^iPhi_in_C}) and then apply the same argument.
\end{proof}

\section{Variational representation and convexity}\label{s.convex}

We refer to Section~\ref{s.org} for the outline of this section.
Let $(B(t))_{t\geq 0}$ be the standard Wiener process in $\R^\D$. For $0\leq s< t$, Let $\mathcal{D}[s,t]$ be the space of $\R^\D$-valued processes $(u(r))_{r\in[s,t]}$ measurable with respect to the sigma-algebra generated by $(B(r))_{r\in[s,t]}$. We equip $\mathcal{D}[s,t]$ with the following norm
\begin{align}\label{e.norm}
    \|u\| = \Ll(\E \int_s^t |u(r)|^2\d r\Rr)^\frac{1}{2}.
\end{align}
Throughout this section, we fix $P_1$ and $\Psi$ and we write, for $\alpha\in\cM$,
\begin{align}\label{e.Psi_alpha=Psi_PPsialpha}
    \Phi_\alpha = \Phi_{P_1,\Psi,\alpha}
\end{align}
given in Definition~\ref{d.sol_pde}.

\begin{remark}
As commented in Remark~\ref{r.condition_PDE}, in Section~\ref{s.PDE}, we only assume that $P_1$ satisfies \eqref{e.supp_P_1} and $\Psi\in\Pi_\mathrm{Lip}(z)$ satisfies \eqref{e.weak_gamma>0}.
We continue to assume these two conditions in this section. Whenever \eqref{e.not_Dirac} for $P_1$ and \eqref{e.gamma>0} for $\Psi$ are needed, we mention them explicitly in the statements.
\end{remark}

We are careful with the condition~\eqref{e.gamma>0} because this is not satisfied in the Potts model with quadratic interaction.

\subsection{Variational formula for the solution}
For $\alpha\in\cM$, $0\leq s<t\leq 1$, $x\in\R^\D$, and $u\in \mathcal{D}[s,t]$, define
\begin{align}\label{e.F=}
    F^{s,t}_\alpha(u,x) = \E \left[C^{s,t}_\alpha(u,x)-L^{s,t}_\alpha(u)\right]
\end{align}
where
\begin{gather}
    C^{s,t}_\alpha(u,x) = \Phi_\alpha\left(t,\,x+\int_s^t\alpha(r)\gamma(r)u(r)\d r + \int_s^t\sqrt{\gamma(r)}\d B(r)\right), \label{e.C=}
    \\
    L^{s,t}_\alpha(u)=\frac{1}{2}\int_s^t\alpha(r)\la\gamma(r),u(r)u(r)^\intercal\ra_{\S^D}\d r.\label{e.L=}
\end{gather}
In $\alpha(r)\gamma(r)u(r)$, $\alpha(r)$ is real-valued serving as a scalar, $\gamma(r)$ is a $\D\times\D$ matrix and $u(r)$ is $\R^\D$-valued. The term $\sqrt{\gamma(r)}\d B(r)$ is interpreted as in~\eqref{e.M(s)=}.
\begin{proposition}\label{p.var}
For $\alpha\in\cM$, $0\leq s<t\leq 1$, $x\in\R^\D$, and $u\in\mathcal{D}[s,t]$, 
\begin{align}\label{e.Phi=maxF}
    \Phi_\alpha(s,x)=\max_{u\in\mathcal{D}[s,t]} F^{s,t}_\alpha(u,x),
\end{align}
where the maximum is achieved at $u^\star_\alpha$ given by
\begin{align}\label{e.def_u}
    u^\star_\alpha(r) = \nabla\Phi_\alpha(r, X_\alpha(r)),\quad\forall r \in [s,t]
\end{align}
where $(X_\alpha(r))_{r\in[s,t]}$ is the unique strong solution of the SDE
\begin{align}\label{e.SDE}
    \begin{split}
        \d X_\alpha(r)  &= \alpha(r)\gamma(r)\nabla\Phi_\alpha(r,X_\alpha(r))\d r + \sqrt{\gamma(r)}\d B(r),
    \\
    X_\alpha(s)  &= x.
    \end{split}
\end{align}
\end{proposition}
Notice that $X_\alpha$ is $\R^\D$-valued.
We refer to \cite[Definition~2.1 in Chapter~5]{karatzas1991brownian} for the notion of strong solutions. The existence of the solution to~\eqref{e.SDE} is ensured by \cite[Theorem~2.9 in Chapter~5]{karatzas1991brownian} together with the Lipschitzness of $\nabla\Phi_\alpha(r,\cdot)$ given by Proposition~\ref{p.reg_Phi}~\eqref{i.sup|d_xPhi|<infty}.

\begin{proof}
\textit{Step~1}. We verify the lower bound for $\Phi_\alpha(s,x)$ assuming $\alpha\in\cM_\d$. Let $\alpha$ be given in~\eqref{e.alpha=} with parameters $(m_l)_{l=-1}^K$ and $(q_l)_{l=0}^K$ in~\eqref{e.m,q}. For simplicity, we further assume that $s,t$ are discontinuity points of $\alpha$.

For $l\in \{0,1,\dots,K\}$, using~\eqref{e.M(t)-M(s)}, we can rewrite~\eqref{e.Phi(s,x)} into 
\begin{align*}
    \Phi_\alpha(q_{l-1},x) = \frac{1}{m_{l-1}} \log \E \exp\left(\m_{l-1} \Phi_\alpha\left(q_l, x+ \int_{q_{l-1}}^{q_l}\sqrt{\gamma(r)}dB(r)\right)\right).
\end{align*}
Define
\begin{align*}
    Z_l &= -\frac{1}{2}\int_{q_{l-1}}^{q_l}\m^2_l\la\gamma(r),u(r)u(r)^\intercal\ra_{\S^D}\d r-\int_{q_{l-1}}^{q_l}\m_{l-1}\la u(r), \sqrt{\gamma(r)}\d B(r)\ra_{\R^\D},
    \\
    \tilde B(r) &= \int_{q_{l-1}}^r\m_{l-1}\sqrt{\gamma(a)}u(a)\d a + B(r),
\end{align*}
and $\tilde\E[\cdots] = \E\Ll[\cdots e^{Z_l}\Rr]$. By the Girsanov theorem \cite[Theorem~5.1 in Chapter~3]{karatzas1991brownian}, $(\tilde B(r))_{r\in[q_{l-1},q_l]}$ is the standard Wiener process in $\R^\D$ under $\tilde \E$. Then, we have
\begin{align*}
    &\Phi_\alpha(q_{l-1},x) = \frac{1}{m_{l-1}} \log \tilde \E \exp\left(\m_{l-1} \Phi_\alpha\left(q_l, x+ \int_{q_{l-1}}^{q_l}\sqrt{\gamma(r)}d\tilde B(r)\right)\right)
    \\
    &=\frac{1}{m_{l-1}} \log  \E  \exp\left(\m_{l-1} \Phi_\alpha\left(q_l, x+ \int_{q_{l-1}}^{q_l}\m_{l-1}\gamma(r)u(r)\d r  + \int_{q_{l-1}}^{q_l}\sqrt{\gamma(r)}d B(r)\right)+Z_l\right).
\end{align*}
Jensen's inequality implies $\log \E \exp(\cdots)\geq \E(\cdots)$. Also note that, for $r\in [q_{l-1},q_l)$, we have $\m_{l-1} = \alpha(r)$. By this, we have
\begin{align*}
    \Phi(q_{l-1},x) \geq F^{q_{l-1},q_l}(u,x).
\end{align*}
Recall that we have assumed $s,t\in \{q_0,\dots,q_K\}$.
Iterating this procedure, we obtain the lower bound
\begin{align}\label{e.Phi_lower}
    \Phi_\alpha(s,x)\geq\sup_{u\in\mathcal{D}[s,t]}\left\{F^{s,t}_\alpha(u,x)\right\}.
\end{align}
The case where one of $s,t$ is a continuity point of $\alpha$ can be proved in the same way.

\textit{Step~2}. We upgrade the lower bound to general $\alpha\in\cM$. Fix any $u\in\mathcal{D}[s,t]$.
Let $(\alpha_n)_{n\in\N}$ be a sequence in $\cM_\d$ that converges to $\alpha$ almost everywhere on $[0,1]$. Such sequence can be extracted from a sequence converging to $\alpha$ in $\cM$.
Hence, $\int_s^t\alpha_n\gamma u\d r$ and $\int_s^t \alpha_n \la \gamma,uu^\intercal\ra_{\S^\D} \d r$ converge respectively to $\int_s^t\alpha\gamma u\d r$ and $\int_s^t \alpha \la \gamma,uu^\intercal\ra_{\S^\D} \d r$ almost surely. 
This along with Proposition~\ref{p.reg_Phi}~\eqref{i.sup|d_xPhi|<infty} and Lemma~\ref{l.Lipschitz_phi} implies that $F^{s,t}_{\alpha_n}(u,x)$ converges to $F^{s,t}_\alpha(u,x)$. Then,~\eqref{e.Phi_lower} follows from this and Lemma~\ref{l.Lipschitz_phi}.

\textit{Step~3}. We show that $u^\star_\alpha$ is a maximizer if $\alpha$ is continuous. In the following until~\eqref{e.step_3}, we write $\Phi=\Phi_\alpha$, $u^\star=u^\star_\alpha$, and $X=X_\alpha$. 
We introduce the following real-valued process
\begin{align*}
    Y(r) = \Phi(r,X(r)) - \frac{1}{2}\int_s^r \alpha(v)\la \gamma(v),u^\star(v)u^\star(v)^\intercal\ra_{\S^D}\d v 
    - \int_s^r \la u^\star(v),\sqrt{\gamma(v)}\d B(v)\ra_{\R^\D}.
\end{align*}
For brevity, we write $\Phi = \Phi(\cdot,X(\cdot))$ in the next display. Allowed by the regularity of $\Phi$ in Proposition~\ref{p.reg_Phi}, we can use It\^o's formula to compute
\begin{align*}
    \d \Phi & \stackrel{\eqref{e.SDE}}{=} (\partial_s\Phi)\d r + \la \nabla\Phi, \d X\ra_{\R^\D} + \frac{1}{2}\la \nabla^2\Phi,\gamma\ra_{\S^D}\d r
    \\
    & \stackrel{\eqref{e.SDE}}{=} \left(\partial_s \Phi + \alpha \la \gamma, (\nabla\Phi)(\nabla\Phi)^\intercal\ra_{\S^D}+\frac{1}{2}\la \gamma,\nabla^2\Phi\ra_{\S^D}\right)\d r + \la \nabla\Phi,\sqrt{\gamma}\d B\ra_{\R^\D}
    \\
    & \stackrel{\eqref{e.PDE_cts}}{=} \frac{1}{2}\alpha\la \gamma,(\nabla\Phi)(\nabla\Phi)^\intercal\ra_{\S^D}\d r+ \la \nabla\Phi,\sqrt{\gamma}\d B\ra_{\R^\D}
    \\
    & \stackrel{\eqref{e.def_u}}{=} \frac{1}{2}\alpha\la \gamma,(u^\star)(u^\star)^\intercal\ra_{\S^D}\d r+ \la u^\star,\sqrt{\gamma}\d B\ra_{\R^\D}.
\end{align*}
Using this, we have that $\d Y =0$. This along with the fact that $Y(s) =\Phi(s,x)$ and the easy observation $\E Y(t) = F^{s,t}_\alpha(u^\star,x)$ yields
\begin{align}\label{e.step_3}
    \Phi_\alpha(s,x) = F^{s,t}_\alpha(u^\star,x).
\end{align}

\textit{Step~4}. We upgrade the previous step to an arbitrary $\alpha\in\cM$. We choose a sequence $(\alpha_n)_{n\in\N}$ of continuous paths in $\cM$ which converges to $\alpha$ almost everywhere on $[0,1]$. 
We write $X_n = X_{\alpha_n}$ and $X=X_\alpha$.
Then, the expression of the SDE~\eqref{e.SDE} gives
\begin{align*}
    \Ll|X_n(r) - X(r)\Rr| 
    &= \Ll|\int_s^r \alpha_n(a)\gamma(a) \nabla \Phi_{\alpha_n}(a,X_n(a))\d a - \int_s^r \alpha(a)\gamma(a)\nabla\Phi_\alpha(a,X(a))\d a\Rr|
    \\
    &\leq \int_s^r \Ll|(\alpha_n(a)-\alpha(a))\gamma(a) \nabla \Phi_{\alpha_n}(a,X_n(a))\Rr|\d a
    \\
    &+ \int_s^r \Ll|\alpha(a)\gamma(a)\Ll(\nabla\Phi_{\alpha_n}(a,X_n(a))-\nabla\Phi_\alpha(a,X(a))\Rr)\Rr|\d a
    \\
    &\leq o_n(1) + C_1\int_s^t \Ll|\nabla\Phi_{\alpha_n}(a,X_n(a))-\nabla\Phi_\alpha(a,X(a))\Rr|\d a
\end{align*}
where in the last line we used the convergence of $\alpha_n$ and Proposition~\ref{p.reg_Phi}~\eqref{i.sup|d_xPhi|<infty} to get $o_n(1)$ and the boundedness of $\alpha\gamma$ (due to \ref{i.xi_loc_lip} and $\Psi \in \Pi_\mathrm{Lip}(z)$) to get $C_1$.
Now, using Proposition~\ref{p.reg_Phi}~\eqref{i.dxPhi_cvg} and~\eqref{i.sup|d_xPhi|<infty}, we can get, for some constant $C_2$,
\begin{align*}
    &\int_s^t \Ll|\nabla\Phi_{\alpha_n}(a,X_n(a))-\nabla\Phi_\alpha(a,X(a))\Rr|\d a
    \\
    &\leq o_n(1)+\int_s^t \Ll|\nabla\Phi_\alpha(a,X_n(a))-\nabla\Phi_\alpha(a,X(a))\Rr|\d a
    \\
    &\leq o_n(1) + C_2 \int_s^t \Ll|X_n(a)-X(a)\Rr|\d a.
\end{align*}
By Gr\"onwall's inequality, we can deduce 
\begin{align*}
    \sup_{r\in[s,t]}|X_n(r)-X(r)|\leq C_3 o_n(1)
\end{align*}
almost surely for some constant $C_3$, where $o_n(1)$ is nonrandom.
Using this and Proposition~\ref{p.reg_Phi}~\eqref{i.d_xPhi_exists} and~\eqref{i.dxPhi_cvg}, we can get from~\eqref{e.def_u} 
\begin{align*}\sup_{r\in[s,t]}\Ll|u^\star_{\alpha_n}(r) - u^\star_\alpha(r)\Rr|\leq C_4 o_n(1)
\end{align*}
almost surely for some deterministic $C_4$ and nonrandom $o_n(1)$.
By this and the convergence of $\alpha_n$, we can deduce $\lim_{n\to\infty}F^{s,t}_{\alpha_n}(u^\star_{\alpha_n},x) = F^{s,t}_\alpha(u^\star_\alpha,x)$. Then, we pass~\eqref{e.step_3} for $\alpha_n$ to the limit to get~\eqref{e.step_3} for $\alpha$. This completes the proof.
\end{proof}

\subsection{Convexity of the solution}

We want to understand the convexity of the Parisi PDE solution.
We start with computation of derivatives, for which we need some notation.
For any function $f:\R^\D\to\R$ and $i_1,\,i_2,\,\dots,\, i_p\in \{1,\dots,\D\}$ for some $p\in\N$, we write
\begin{align}\label{e.f_iiii=...}
    f_{i_1i_2\cdots i_p}=\partial_{x_{i_1}}\partial_{x_{i_2}}\cdots \partial_{x_{i_p}}f.
\end{align}
The following generalizes \cite[Lemma~2]{aufche} to the vector case.
\begin{lemma}\label{l.deriv_Phi}
For $\alpha\in\cM$, $0\leq s\leq t\leq 1$, and $x\in\R^\D$, let $X=X_\alpha$ be given in~\eqref{e.SDE}. Then, for $s\leq a \leq b\leq t$ and $\Phi=\Phi_\alpha$, the following holds:
\begin{align*}
    \nabla\Phi(a,X(b))-\nabla\Phi(a,X(a)) & = \int_a^b \nabla^2\Phi \sqrt{\gamma}\d B;
    \\
    \Phi_{kl}(b,X(b)) - \Phi_{kl}(a,X(a)) & = \int_a^b \sum_{i,j=1}^\D\Ll( - \alpha \gamma_{ij}\Phi_{ik}\Phi_{jl}\d r +\Phi_{ikl}\Ll(\sqrt{\gamma}\Rr)_{ij}\d B_j \Rr)
\end{align*}
for all $k,l\in\{1,\dots,\D\}$.
\end{lemma}

In the first line of the display, $\nabla^2\Phi \sqrt{\gamma}$ is the matrix multiplication between the Hessian $\nabla^2\Phi$ and $\sqrt{\gamma}$; and $\nabla^2\Phi \sqrt{\gamma}\d B$ as the multiplication of the matrix $\nabla^2\Phi \sqrt{\gamma}$ by the vector $\d B$. The second line is written in the notation introduced in~\eqref{e.f_iiii=...}.

\begin{proof}
By an approximation argument similar to that in Step~4 of the proof of Proposition~\ref{p.var}, we can assume that $\alpha$ is continuous.
We rewrite~\eqref{e.PDE_cts} as
\begin{align*}
    \partial_s\Phi + \frac{1}{2}\sum_{i,j=1}^\D \gamma_{ij}\Ll(\Phi_{ij} + \Phi_i\Phi_j\Rr) = 0.
\end{align*}
Allowed by Proposition~\ref{p.reg_Phi},
using the product rule and the symmetry of $\gamma$, we can compute
\begin{gather*}
    \partial_s\Phi_k + \sum_{i,j=1}^\D \gamma_{ij}\Ll(\frac{1}{2}\Phi_{ijk} + \alpha\Phi_{ik}\Phi_j\Rr) = 0, \\
    \partial_s\Phi_{kl} + \sum_{i,j=1}^\D \gamma_{ij}\Ll(\frac{1}{2}\Phi_{ijkl} + \alpha\Phi_{ikl}\Phi_j+\alpha\Phi_{ik}\Phi_{jl}\Rr) = 0\end{gather*}
In the following, the notation~\eqref{e.f_iiii=...} only applies to $\Phi$. We use the shorthand $\Phi(r) = \Phi(r,X(r))$.
Writing $X=(X_i)_{i=1}^\D$, we can see from~\eqref{e.SDE} that
\begin{align*}
    \d X_i = \alpha \sum_{j=1}^\D\Ll(\gamma_{ij}\Phi_j \d r + \Ll(\sqrt{\gamma}\Rr)_{ij} \d B_j\Rr).
\end{align*}
Using the It\^o formula and the computations above, we have
\begin{align*}
    \d \Phi_k & = \partial_s \Phi_k \d r + \sum_{i=1}^\D \Phi_{ik}\d X_i + \frac{1}{2}\sum_{i,j=1}^\D \Phi_{ijk}\gamma_{ij}\d r
    \\
    & =- \sum_{i,j=1}^\D\Ll(\frac{1}{2}\gamma_{ij}\Phi_{ijk}+\alpha \gamma_{ij}\Phi_{ik}\Phi_j \Rr)\d r
     + \sum_{i,j=1}^\D \Ll(\Phi_{ik}\alpha\gamma_{ij}\Phi_{j} \d r+\Phi_{ik}\Ll(\sqrt{\gamma}\Rr)_{ij}\d B_j\Rr)
    \\
    & + \frac{1}{2}\sum_{i,j}^\D \Phi_{ijk}\gamma_{ij} \d r
     = \sum_{i,j=1}^\D \Phi_{ik}\Ll(\sqrt{\gamma}\Rr)_{ij}\d B_j,
\end{align*}
and
\begin{align*}
    \d \Phi_{kl} & = \partial_s \Phi_{kl} \d r + \sum_{i=1}^\D \Phi_{ikl}\d X_i + \frac{1}{2}\sum_{i,j=1}^\D \Phi_{ijkl}\gamma_{ij}\d r
    \\
    & = -\sum_{i,j=1}^\D\Ll( \frac{1}{2}\gamma_{ij}\Phi_{ijkl}+\alpha\gamma_{ij}\Phi_{ikl}\Phi_j + \alpha \gamma_{ij}\Phi_{ik}\Phi_{jl}\Rr)\d r 
    \\
    &+ \sum_{i,j=1}^\D \Ll( \Phi_{ikl}\alpha\gamma_{ij}\Phi_j \d r+\Phi_{ikl}\Ll(\sqrt{\gamma}\Rr)_{ij}\d B_j\Rr)
    + \frac{1}{2}\sum_{i,j=1}^\D \Phi_{ijkl}\gamma_{ij}\d r
    \\
    & = \sum_{i,j=1}^\D\Ll(- \alpha \gamma_{ij} \Phi_{ik}\Phi_{jl} + \Phi_{ikl}\Ll(\sqrt{\gamma}\Rr)_{ij}\d B_j\Rr).
\end{align*}
The announced equations immediately follow from the above.
\end{proof}

Results on the convexity of the Parisi PDE solution are summarized below.
\begin{lemma}
\label{l.Phi_convex}
The following holds:
\begin{enumerate}
    \item\label{i.Phi_alpha_convex} $\Phi_\alpha(s,\cdot)$ is convex for every $\alpha\in\cM$ and $s\in[0,1]$;
    \item\label{i.d^2Phi>0} if $P_1$ satisfies \eqref{e.not_Dirac}, then there is $y\in\R^\D$ such that $y\cdot \nabla^2 \Phi_\alpha(s,x)y>0$ for every $\alpha\in\cM$ and $(s,x)\in [0,1]\times \R^\D$;
    \item \label{i.d^2Phi<C}
    there is a constant $C$ such that
    $\nabla^2\Phi_\alpha(s,x) \cleq C\identity_\D$ for every $\alpha\in\cM$ and $(s,x)\in [0,1]\times \R^\D$.
\end{enumerate}
\end{lemma}

\begin{proof}
Fix any $\alpha$ and write $\Phi=\Phi_\alpha$.
We first prove parts~\eqref{i.Phi_alpha_convex} and~\eqref{i.d^2Phi>0} at $s=1$ for $\Phi(1,\cdot)=\phi$.
For any $y\in\R^\D$, we can compute
\begin{align}\label{e.yd^2phiy}
    y\cdot \nabla^2\phi(x) y = \frac{\d^2}{\d\eps^2} \phi(x+\eps y)\Big|_{\eps =0}= \la \Ll(\sigma\cdot y - \la \sigma\cdot y\ra_x\Rr)^2\ra_x \geq 0
\end{align}
where $\la\cdot\ra_x$ is given in~\eqref{e.<>_x}. This verifies part~\eqref{i.Phi_alpha_convex} at $s=1$. 

To show part~\eqref{i.d^2Phi>0} at $s=1$, we argue by contradiction and assume that for every $y\in\R^\D$ there is $x_y\in \R^\D$ such that $y\cdot \nabla^2 \phi(x_y)y=0$. Setting $r_y = \la \sigma\cdot y\ra_{x_y}$, from~\eqref{e.yd^2phiy}, we have $\sigma\cdot y = r_y$ a.s.\ under $\la\cdot\ra_{x_y}$. Since $e^{\sigma\cdot x_y}$ in~\eqref{e.<>_x} is strictly positive, we must have $\sigma\cdot y =r_y$ a.e.\ under $P_1$. Varying $y$, we deduce that $\sigma$ is constant a.e.\ under $P_1$ and thus $P_1$ is a Dirac measure which contradicts~\eqref{e.not_Dirac}. Hence, part~\eqref{i.d^2Phi>0} holds at $s=1$.

Next, we prove parts~\eqref{i.Phi_alpha_convex} and~\eqref{i.d^2Phi>0} for $s\in[0,1)$.
We apply Lemma~\ref{l.deriv_Phi} with $a=s$ and $b=t=1$ so that $\Phi(a,X(a))=\Phi(s,x)$ and $\Phi(b,X(b))=\phi(X(1))$. Hence, for any $y\in\R^\D$, we get
\begin{align}
    y\cdot \nabla^2\Phi(s,x)y &= \sum_{k,l=1}^\D y_k\Phi_{kl}(s,x)y_l \notag
    \\
    &= \E\Ll[\sum_{k,l=1}^\D y_k\phi_{kl}(X(1))y_l + \int_s^1\sum_{i,j,k,l=1}^\D \alpha\gamma_{ij}\Phi_{ik}\Phi_{jl}y_ky_l \d r\Rr] \notag
    \\
    & = \E \Ll[y\cdot \nabla^2\phi(X(1))y+\int_s^1 \alpha\gamma\cdot \Ll(\nabla^2\Phi y\Rr)\Ll(\nabla^2\Phi y\Rr)^\intercal \d r\Rr]\label{e.ynabla^2Phiy>}
\end{align}
where inside the expectation we used the short hand $\Phi(r) = \Phi(r,X(r))$.
Due to $\gamma(s)\in \S^\D_+$ for every $s$ (by~\ref{i.xi_incre} and $\Psi\in \Pi_\mathrm{Lip}(z)$) and~\eqref{e.ab>0}, the last integral is nonnegative a.s.
Hence, part~\eqref{i.Phi_alpha_convex} follows from the special case $s=1$.
Let $y$ be given in part~\eqref{i.d^2Phi>0} at $s=1$. Then, the above display implies that part~\eqref{i.d^2Phi>0} holds at $s\in[0,1)$.

It remains to verify part~\eqref{i.d^2Phi<C}. We show
\begin{align}\label{e.a<..I_D}
    a \cleq \sqrt{D}|a|\identity_\D,\quad\forall a\in \S^\D.
\end{align}
Let $a = qbq^\intercal$ be the diagonalization of $a$, where $b\in\S^\D$ is diagonal and $q\in\R^{\D\times \D}$ is orthogonal. For every $i\in\{1,\dots,\D\}$,
\begin{align*}
    b_{ii} \leq \sqrt{\sum_{j=1}^\D (bb)_{jj}} \leq \sqrt{\sum_{j=1}^\D \tr(bb)} = \sqrt{D\tr(bb)}=\sqrt{D\tr(aa)}=\sqrt{D}|a|.
\end{align*}
Hence, we deduce $x^\intercal bx\leq \sqrt{D}|a| x^\intercal\identity_\D x$ for every $x\in\R^\D$, which implies~\eqref{e.a<..I_D}. Then, part~\eqref{i.d^2Phi<C} follows from~\eqref{e.a<..I_D} and Proposition~\ref{p.reg_Phi}~\eqref{i.sup|d_xPhi|<infty}. 
\end{proof}

\subsection{Uniqueness of the minimizer}

\begin{lemma}\label{l.unique_maximizer}
Assume~\eqref{e.gamma>0}.
Let $\alpha\in\cM$, $0\leq s<t\leq 1$, and $C$ be given in Lemma~\ref{l.Phi_convex}~\eqref{i.d^2Phi<C}. If $\alpha(s)>0$ and
\begin{align}\label{e.int_s^ttr<c}
    \int_s^t \alpha\tr(\gamma)\d r <C^{-1},
\end{align}
then, for every $x\in \R^\D$, $u^\star_\alpha$ given in~\eqref{e.def_u} is the unique maximizer of~\eqref{e.Phi=maxF} over $\mathcal{D}[s,t]$.
\end{lemma}

\begin{proof}
We show that $u\mapsto F^{s,t}_\alpha(u,x)$ is strictly concave.
Let $u_0,\, u_1\in \mathcal{D}[s,t]$ be distinct and thus, by the definition of the norm on $\mathcal{D}[s,t]$ in~\eqref{e.norm}, we have
\begin{align}\label{e.Eint|u-u|>0}
    \E \int_s^t \Ll|u_1(r)-u_0(r)\Rr|^2\d r >0.
\end{align}
For $\lambda\in [0,1]$, we set $u_\lambda = (1-\lambda) u_0 + \lambda u_1$. Using the definition of $F_\alpha^{s,t}$ in~\eqref{e.F=}, we can compute
\begin{align*}
    \frac{\d }{\d \lambda} F^{s,t}_\alpha(u_\lambda, x) = \E \la \nabla \Phi_\alpha(t,\cdots),\,\int_s^t\alpha\gamma(u_1-u_0)\d r \ra_{\R^\D}
    - \E \int_s^t \alpha \la \gamma,\, u_\lambda(u_1-u_0)^\intercal\ra_{\S^\D} \d r
\end{align*}
where in $\cdots$ we omitted $x+\int_s^t\alpha\gamma u_\lambda \d r + \int_s^t\sqrt{\gamma}\d B$.
Differentiate it once more to get
\begin{align*}
    \frac{\d^2}{\d\lambda^2} F^{s,t}_\alpha(u_\lambda,x) = \E \la\nabla^2\Phi_\alpha(t,\cdots),\,\Ll(\int_s^t\alpha\gamma(u_1-u_0)\d r\Rr)\Ll(\int_s^t\alpha\gamma(u_1-u_0)\d r\Rr)^\intercal \ra_{\S^\D}
    \\
    - \E \int_s^t \alpha \la \gamma, \, (u_1-u_0)(u_1-u_0)^\intercal \ra_{\S^\D} \d r.
\end{align*}
By Lemma~\ref{l.Phi_convex}~\eqref{i.d^2Phi<C}, the first expectation is bounded from above by
\begin{align*}
    C\E \Ll|\int_s^t \alpha \gamma(u_1-u_0)\d r\Rr|^2.
\end{align*}
Writing $a = (\alpha\gamma)^\frac{1}{2}$ and $v= u_1-u_0$ and using the Cauchy--Schwarz inequality, we have
\begin{align*}
    \Ll|\int_s^t \alpha \gamma(u_1-u_0)\d r\Rr|^2
    = \sum_{i=1}^\D\Ll(\int_s^t \sum_{j=1}^\D a_{ij}(av)_j\d r\Rr)^2
    \\
    \leq \sum_{i=1}^\D \Ll(\int_s^t \sum_{j=1}^\D a^2_{ij}\d r\Rr)\Ll(\int_s^t \sum_{j=1}^\D \Ll((av)_j\Rr)^2\d r\Rr)
    \\
    =\Ll(\int_s^t \sum_{i,\,j=1}^\D a^2_{ij}\d r\Rr)\Ll(\int_s^t \sum_{j=1}^\D \Ll((av)_j\Rr)^2\d r\Rr) 
    \\
    =\Ll(\int_s^t \tr(aa)\d r\Rr) \Ll(\int_s^t \la a^2,\, vv^\intercal\ra_{\S^\D}\d r\Rr).
\end{align*}
Combining the above estimates, we obtain
\begin{align*}
    \frac{\d^2}{\d\lambda^2} F^{s,t}(u_\lambda,x) \leq  \Ll(C\int_s^t \alpha\tr(\gamma)\d r -1\Rr) \E \int_s^t \alpha \la \gamma,\, (u_1-u_0)(u_1-u_0)^\intercal \ra_{\S^\D} \d r.
\end{align*}
Since $\alpha(s)>0$ and $\alpha$ is increasing we have $\alpha>c_1$ on $[s,t]$ for some $c_1>0$. The assumption~\eqref{e.gamma>0} along with~\eqref{e.Eint|u-u|>0} ensures that the expectation in the above display is strictly positive.
Then,~\eqref{e.int_s^ttr<c} makes the right-hand side of the display strictly negative. Hence, $F^{s,t}_\alpha(\cdot,x)$ is strictly concave on $\mathcal{D}[s,t]$ and thus the maximizer must be unique. 
\end{proof}

\subsection{Strict functional convexity}
We investigate the convexity of $\alpha\mapsto \Phi_\alpha$.
\begin{lemma}\label{l.Phi_convex_alpha}
Let $\alpha_0,\,\alpha_1\in\cM$. For every $s\in[0,1]$, $\lambda \in [0,1]$, and every $x_0,\,x_1\in \R^\D$,
\begin{align*}
    \Phi_{\alpha_\lambda}(s,x_\lambda) \leq(1-\lambda)\Phi_{\alpha_0}(s,x_0)+\lambda \Phi_{\alpha_1}(s,x_1)
\end{align*}
where $\alpha_\lambda = (1-\lambda)\alpha_0+ \lambda \alpha_1$ and $x_\lambda = (1-\lambda)x_0+\lambda x_1$.
\end{lemma}

\begin{proof}
Recall definitions in~\eqref{e.F=}--\eqref{e.L=}.
For any $u\in \mathcal{D}[s,1]$, it is immediate that
\begin{align*}
    L^{s,1}_{\alpha_\lambda} (u) = (1-\lambda)L^{s,1}_{\alpha_0}(u)+ \lambda L^{s,1}_{\alpha_1}(u).
\end{align*}
Since $\Phi_{\alpha_\lambda}(1,\cdot)= \phi$ for every $\lambda$, the convexity of $\phi$ due to Lemma~\ref{l.Phi_convex}~\eqref{i.Phi_alpha_convex} implies
\begin{align*}
    C^{s,1}_{\alpha_\lambda}(u,x_\lambda) \leq (1-\lambda)C^{s,1}_{\alpha_0}(u,x_0) + \lambda C^{s,1}_{\alpha_1}(u,x_1).
\end{align*}
Therefore,
\begin{align*}
    F^{s,1}_{\alpha_\lambda}(u,x_\lambda) \leq (1-\lambda)F^{s,1}_{\alpha_0}(u,x_0) + \lambda F^{s,1}_{\alpha_1}(u,x_1).
\end{align*}
This along with Proposition~\ref{p.var} at $t=1$ finishes the proof.
\end{proof}

\begin{proposition}\label{p.strict_convex}
Assume~\eqref{e.not_Dirac} and~\eqref{e.gamma>0}.
Let $\alpha_0,\,\alpha_1\in\cM$ be distinct and set
\begin{align*}
    \tau = \min \Ll\{s\in [0,1]:\: \alpha_0(r)=\alpha_1(r),\ \forall r\in [s,1]\Rr\}.
\end{align*}
Then, for every $s\in [0,\tau)$, $\lambda\in (0,1)$, and $x_0,\,x_1\in\R^\D$,
\begin{align}\label{e.ineq_strict_convex}
    \Phi_{\alpha_\lambda}(s,x_\lambda) <(1-\lambda)\Phi_{\alpha_0}(s,x_0)+\lambda \Phi_{\alpha_1}(s,x_1)
\end{align}
where $\alpha_\lambda = (1-\lambda)\alpha_0+ \lambda \alpha_1$ and $x_\lambda = (1-\lambda)x_0+\lambda x_1$.
\end{proposition}

\begin{proof}
\textit{Step~1}. 
We choose $\tau'\in(0,\tau)$ as follows.
If $\int_s^\tau \alpha_0(r)\d r = \int_s^\tau \alpha_1(r) \d r=0$ for some $s\in(0,\tau)$, then we must have that $\alpha_0 = \alpha_1 =0$ on $[s,\tau)$ since $\alpha_0$ and $\alpha_1$ are increasing and nonnegative. But this contradicts the definition of $\tau$. Hence, fixing any $s\in (0,\tau]$, without loss of generality, we can assume $\int_s^\tau \alpha_0(r)\d r>0$. Since $\alpha_0$ is increasing, we can choose $\tau'\in(s,\tau]$ sufficiently close to $\tau$ so that
\begin{align}\label{e.choice_tau'}
    \alpha_0(\tau') >0,\qquad \int_{\tau'}^\tau \alpha_0(r)\tr\Ll(\gamma(r)\Rr)\d r<C^{-1},
\end{align}
for $C$ in Lemma~\ref{l.unique_maximizer}.

In the remaining two steps, we will first verify~\eqref{e.ineq_strict_convex} for all $s\in[\tau',\tau)$ and then for $s\in[0,\tau')$. Henceforth, for $\lambda\in[0,1]$, we write $\Phi^\lambda = \Phi_{\alpha_\lambda}$, $u^\star_\lambda = u^\star_{\alpha_\lambda}$ given in~\eqref{e.def_u}, and $X_\lambda = X_{\alpha_\lambda}$ given in~\eqref{e.SDE} with initial condition $x_\lambda$.

\textit{Step~2}. 
We argue by contradiction and suppose
\begin{align}\label{e.contrad_assump}
    \exists s\in [\tau',\tau),\,\lambda \in(0,1):\quad \Phi^\lambda(s,x_\lambda) = (1-\lambda)\Phi^0(s,x_0)+ \lambda \Phi^1 (s,x_1).
\end{align}
In this step, we set $t=\tau$.

Recall definitions in~\eqref{e.F=},~\eqref{e.C=}, and~\eqref{e.L=}. We have
\begin{align}\label{e.linear_L_lambda}
    L^{s,t}_{\alpha_\lambda} (u^\star_\lambda) = (1-\lambda)L^{s,t}_{\alpha_0}(u^\star_\lambda)+ \lambda L^{s,t}_{\alpha_1}(u^\star_\lambda).
\end{align}
This along with Lemma~\ref{l.Phi_convex_alpha} implies 
\begin{align*}
    F^{s,t}_{\alpha_\lambda}(u^\star_\lambda,x_\lambda) \leq (1-\lambda) F^{s,t}_{\alpha_0}(u^\star_\lambda,x_0)+\lambda F^{s,t}_{\alpha_1}(u^\star_\lambda,x_1).
\end{align*}
Using this, ~\eqref{e.contrad_assump}, and relations due to Proposition~\ref{p.var}
\begin{align*}
    F^{s,t}_{\alpha_\lambda}(u^\star_\lambda,x_\lambda)=\Phi^\lambda(s,x_\lambda),\qquad 
    F^{s,t}_{\alpha_0}(u^\star_\lambda,x_0)\leq\Phi^0(s,x_0),\qquad 
    F^{s,t}_{\alpha_1}(u^\star_\lambda,x_1)\leq\Phi^1(s,x_1),
\end{align*}
we get 
\begin{align*}
    F^{s,t}_{\alpha_0}(u^\star_\lambda,x_0)=\Phi^0(s,x_0),\qquad 
    F^{s,t}_{\alpha_1}(u^\star_\lambda,x_1)=\Phi^1(s,x_1),
\end{align*}
and thus $u^\star_\lambda$ is a maximizer for $\Phi^0(s,x_0)$ and $\Phi^1(s,x_1)$.
Since $\alpha$ is increasing and nonnegative, \eqref{e.choice_tau'} holds with $s$ substituted for $\tau'$ therein.
This along with Lemma~\ref{l.unique_maximizer} implies that $u^\star_0$ is the unique maximizer of $\Phi^0(s,x_0)$ and thus $u^\star_0 = u^\star_\lambda$.
Hence, by~\eqref{e.def_u},
\begin{align*}\nabla\Phi^0(x,X_0(r)) = u^\star_0(r) = u^\star_\lambda(r) = \nabla\Phi^\lambda(x,X_\lambda(r)),\quad\forall r \in [s,t].
\end{align*}
In the remaining of this step, we write $\Phi^0 = \Phi^0(\cdot,X_0(\cdot))$ and similarly for $\Phi^\lambda$.
The above display along with Lemma~\ref{l.deriv_Phi} yields
\begin{align*}
    \int_s^t \nabla^2 \Phi^0 \sqrt{\gamma} \d B & = \int_s^t \nabla^2 \Phi^\lambda \sqrt{\gamma} \d B.
\end{align*}
The It\^o isometry implies 
\begin{align*}
    \E\int_s^t \Ll|\nabla^2 \Ll(\Phi^0-\Phi^\lambda\Rr)\sqrt{\gamma}\Rr|^2 \d r =0.
\end{align*}
Due to the assumption~\eqref{e.gamma>0} and the continuity of $r\mapsto \Phi^0$ and $r\mapsto\Phi^\lambda$, we get 
\begin{align}\label{e.nabla^2Phi^0=nabla^2Phi^lambda}
    \nabla^2\Phi^0 = \nabla^2 \Phi^\lambda.
\end{align}
Recall the notation in~\eqref{e.f_iiii=...}.
Using this and Lemma~\ref{l.deriv_Phi}, we have, for $s\leq a\leq b\leq t$ and $k,l\in\{1,\dots,\D\}$,
\begin{gather*}
     -\int_a^b \alpha_0 \E \Ll[ \sum_{i,j=1}^\D \gamma_{ij}\Phi^0_{ik}\Phi^0_{jl}\Rr] \d r  = \E \Ll[ - \int_a^b\sum_{i,j=1}^\D \Ll(-\alpha_0 \gamma_{ij}\Phi^0_{ik}\Phi^0_{jl}\d r + \Phi^0_{ikl}\Ll(\sqrt{\gamma}\Rr)_{ij}\d B_j\Rr)\Rr]
     \\
      = \E \Ll[\Phi^0_{kl}\Ll(b,X_0(b)\Rr) - \Phi^0_{kl}\Ll(a,X_0(a)\Rr)\Rr] 
     = \E \Ll[\Phi^\lambda_{kl}\Ll(b,X_\lambda(b)\Rr) - \Phi^\lambda_{kl}\Ll(a,X_\lambda(a)\Rr)\Rr]
     \\
      = \E \Ll[ - \int_a^b\sum_{i,j=1}^\D \Ll(-\alpha_\lambda \gamma_{ij}\Phi^\lambda_{ik}\Phi^\lambda_{jl}\d r + \Phi^\lambda_{ikl}\Ll(\sqrt{\gamma}\Rr)_{ij}\d B_j\Rr)\Rr]
      \\
      = -\int_a^b \alpha_\lambda \E \Ll[ \sum_{i,j=1}^\D \gamma_{ij}\Phi^\lambda_{ik}\Phi^\lambda_{jl}\Rr] \d r.
\end{gather*}
Let $y$ be given in Lemma~\ref{l.Phi_convex}~\eqref{i.d^2Phi>0} (which requires the assumption~\eqref{e.not_Dirac}).
The above display implies, for $s\leq a\leq b\leq t$,
\begin{align*}
    \int_a^b\alpha_0 \E \Ll[\gamma\cdot \Ll(\nabla^2\Phi^0y\Rr)\Ll(\nabla^2\Phi^0y\Rr)^\intercal\Rr]\d r = \int_a^b\alpha_\lambda \E \Ll[\gamma\cdot \Ll(\nabla^2\Phi^\lambda y\Rr)\Ll(\nabla^2\Phi^\lambda y\Rr)^\intercal\Rr]\d r.
\end{align*}

By Lemma~\ref{l.Phi_convex}~\eqref{i.d^2Phi>0}, $\nabla^2\Phi^0 y$ is a nonzero vector. Due to $\gamma\in \S^\D_{++}0$ a.e.\ assumed in~\eqref{e.gamma>0}, we have $\gamma\cdot \Ll(\nabla^2\Phi^0y\Rr)\Ll(\nabla^2\Phi^0y\Rr)^\intercal>0$ a.e.\ on $[0,1]$ for almost every realization. Using this and~\eqref{e.nabla^2Phi^0=nabla^2Phi^lambda} again, we have that if $\alpha_0(a) \neq \alpha_\lambda(a)$ for some $a\in[s,t)$, then
\begin{align*}
    \int_a^b\alpha_0 \E \Ll[\gamma\cdot \Ll(\nabla^2\Phi^0y\Rr)\Ll(\nabla^2\Phi^0y\Rr)^\intercal\Rr]\d r \neq \int_a^b\alpha_\lambda \E \Ll[\gamma\cdot \Ll(\nabla^2\Phi^\lambda y\Rr)\Ll(\nabla^2\Phi^\lambda y\Rr)^\intercal\Rr]\d r.
\end{align*}
for $b$ sufficiently close to $a$, which contradicts the previous display. Hence, we must have $\alpha_0 = \alpha_1$ on $[s,t]$ but this contradicts the definition of $\tau$ (recall $t=\tau$ in this step). Therefore, \eqref{e.contrad_assump} is invalid and the inequality~\eqref{e.ineq_strict_convex} holds for all $s\in[\tau',\tau)$.

\textit{Step~3}. 
It remains to verify~\eqref{e.ineq_strict_convex} for $s\in[0,\tau')$. Henceforth, we set $t=\tau'$. By~\eqref{e.ineq_strict_convex} at $t$, we have
\begin{align*}
    \Phi^\lambda(t,(1-\lambda)y_0 +\lambda y_1) < (1-\lambda) \Phi^0(t,y_0)+\lambda \Phi^1(t,y_1)
\end{align*}
for any $y_0,\,y_1\in\R^\D$.
This implies
\begin{align*}
    C^{s,t}_{\alpha_\lambda}(u^\star_\lambda,x_\lambda) < (1-\lambda)C^{s,t}_{\alpha_0}(u^\star_\lambda,x_0) + 
    \lambda C^{s,t}_{\alpha_1}(u^\star_\lambda,x_1).
\end{align*}
We still have the relation~\eqref{e.linear_L_lambda}, which together with the above display implies
\begin{align*}
    F^{s,t}_{\alpha_\lambda}(u^\star_\lambda,x_\lambda) < (1-\lambda)F^{s,t}_{\alpha_0}(u^\star_\lambda,x_0) + 
    \lambda F^{s,t}_{\alpha_1}(u^\star_\lambda,x_1).
\end{align*}
Therefore, by Proposition~\ref{p.var},
\begin{align*}
    \Phi^\lambda(s,x_\lambda) = F^{s,t}_\lambda(u^\star_\lambda,x_\lambda) < (1-\lambda)\Phi^0(s,x_0) + 
    \lambda \Phi^1(s,x_1),
\end{align*}
completing the proof.
\end{proof}

Lastly, we comment that the argument by Jagannath and Tobasco in \cite{jagtob} (see the proof of \cite[Theorem~20]{jagtob}) to show the strict convexity may need $\phi$ to be strictly convex everywhere (to verify \cite[Lemma~21]{jagtob}). The approach by Auffinger and Chen allows the weaker condition that there is some $y\in\R^\D$ such that $y\cdot \nabla^2\phi(x)y>0$ for all $x\in\R^\D$, which holds due to Lemma~\ref{l.Phi_convex}~\eqref{i.d^2Phi>0}.

\section{Proofs of main results}\label{s.pf_main}

We start by defining the Parisi functional appearing in~\eqref{e.parisi}.
Recall the Ruelle probability cascade measure $\mathfrak{R}$ described above~\eqref{e.w^pi} and the Gaussian process $(w^\pi(\brho))_{\brho\in\supp\mathfrak{R}}$ in~\eqref{e.w^pi}.
Recall $\Pi$, $\Pi(z)$, $\Pi_\mathrm{Lip}(z)$, and $\Pi_\mathrm{reg}(z)$ in~\eqref{e.Pi},~\eqref{e.Pi(z)},~\eqref{e.Pi_Lip(z)}, and~\eqref{e.Pi_reg}.
For $\pi\in\Pi$ and $h\in\S^\D$, we define
\begin{align*}\sP(\pi)= \E \log\iint \exp\Ll(w^{\nabla\xi\circ\pi}(\alpha)\cdot \sigma - \frac{1}{2}\nabla\xi\circ\pi(1)\cdot \sigma\sigma^\intercal\Rr)\d P_1(\sigma)\d \mathfrak{R}(\alpha) 
    +\frac{1}{2}\int_0^1\theta(\pi(s))\d s 
\end{align*}
where $\theta$ is in \eqref{e.theta}.
Note that \ref{i.xi_incre} ensures $\nabla\xi\circ \pi \in \Pi$.
In the above, if $\pi\in \Pi(z)$, then $\pi(1)=z$.
We set $\Pi_\d(z) = \{\pi\in\Pi(z)\ \big|\ \text{$\pi$ is a step function}\}$.

Next, we relate $\sP$ to the functional $\sF$ in~\eqref{e.sF}.
As commented in Remark~\ref{r.condition_PDE}, it is only assumed in Section~\ref{s.PDE} that $P_1$ satisfies~\eqref{e.supp_P_1} and that $\Psi$ satisfies~\eqref{e.weak_gamma>0}. It is clear that if $P_1$ satisfies~\eqref{e.supp_P_1}, then so does $\tilde P_1^z$ in~\eqref{e.tildeP_1}. Hence, we can use results in Section~\ref{s.PDE} to relate the two functionals.
\begin{lemma}
\label{l.sF=sP}
Let $P_1$ satisfy~\eqref{e.supp_P_1} and $\Psi\in \Pi_\mathrm{Lip}(z)$ satisfy \eqref{e.weak_gamma>0}. If $\alpha\in\cM_\d$, then $\Psi\circ\alpha^{-1}\in \Pi_\d (z)$ and $\sF(\Psi,\alpha) = \sP\Ll(\Psi\circ \alpha^{-1}\Rr)$.

\end{lemma}
\begin{proof}
We assume that $\alpha$ takes the form in~\eqref{e.alpha=} and then $\alpha^{-1}$ is given in~\eqref{e.alpha^-1=}.
Let us write $\pi = \Psi\circ\alpha^{-1}$.
It is clear that $\pi$ is a step function. 
From~\eqref{e.alpha^-1=}, we have $\alpha^{-1}(1)=1$, which along with $\Psi(1)=z$ implies $\pi(1) =z$. Thus, $\pi\in\Pi_\d(z)$. 
Using Corollary~\ref{c.EPhi=} where $\mu = \nabla\xi\circ\Psi$ is the notation introduced in~\eqref{e.mu}, we get
\begin{align*}
    \E\Ll[\Phi_{\tilde P_1^z,\Psi,\alpha}\Ll(0,\sqrt{\nabla\xi\circ\Psi(0)}\eta\Rr)\Rr]=\E\log\iint \exp\Ll(\sigma\cdot w^{\nabla\xi\circ\pi}(\brho)\Rr)\d \tilde P_1^z(\sigma)\d\mathfrak{R}(\brho)
\end{align*}
which is equal to the first term in $\sP(\pi)$ due to $\pi(1)=z$. Then, we verify that the remaining terms in $\sF(\Psi,\alpha)$ match the second term in $\sP(\pi)$.

Recall $\mu$ in~\eqref{e.mu}.
The product rule yields
\begin{align*}
    \frac{\d}{\d s}(\Psi\cdot \mu) = \dot\Psi\cdot \mu + \Psi\cdot \dot\mu =  \frac{\d}{\d s}(\xi\circ\Psi) + \Psi\cdot \gamma.
\end{align*}
Recall the definition of $\theta$ in~\eqref{e.theta}, and we have $\theta\circ\Psi=\Psi\cdot \mu - \xi\circ\Psi$.
Using these, we can compute
\begin{align*}
    \int_0^1 \alpha(s)\Psi(s)\cdot \gamma(s) \d s & = \int_0^1 \Psi(s)\cdot \gamma(s)\int_0^s \d \alpha(r)\d s = \int_0^1\int_r^1 \Psi(s)\cdot \gamma(s)\d s \d \alpha(r)
    \\
    &=\int_0^1\int_r^1\Ll(\frac{\d}{\d s}\Ll(\Psi(s)\cdot \mu(s)\Rr) -\frac{\d}{\d s}(\xi\circ\Psi)(s)\Rr)\d s \d \alpha(r)
    \\
    &=\int_0^1 \Ll(\Psi(1)\cdot\mu(1)-\xi\circ\Psi(1)-\Psi(r)\cdot\mu(r)+\xi\circ\Psi(r)\Rr)\d \alpha(r)
    \\
    & = \theta(\Psi(1)) - \int_0^1\theta(\Psi(r))\d \alpha(r).
\end{align*}
Since $\Psi(1)=z$, this matches the remaining terms and completes the proof.
\end{proof}

We show that assumptions in Theorems~\ref{t.lim} and~\ref{t.convex} make~\eqref{e.weak_gamma>0} valid.

\begin{lemma}\label{l.weak->gamma>0}
\hfill
\begin{enumerate}
    
    \item \label{i.l.weak->gamma>0_2} If $\xi$ is twice continuously differentiable and satisfies \eqref{e.cond_xi}, then $\Psi$ satisfies \eqref{e.gamma>0} for every $\Psi\in \Pi_\mathrm{reg}(z)$.

    \item \label{i.l.weak->gamma>0_1} If $\Psi\in \Pi_\mathrm{Lip}(z)$ satisfies~\eqref{e.gamma>0}, then $\Psi$ satisfies~\eqref{e.weak_gamma>0}.
\end{enumerate}
\end{lemma}

\begin{proof}
Part~\eqref{i.l.weak->gamma>0_2}.
Let $\xi$ satisfy the conditions. 
For every $b\in\S^\D_+\setminus\{0\}$, we can compute
\begin{align*}
    b\cdot \gamma(s) = \frac{\d}{\d s}\Ll(b\cdot \nabla\xi(\Psi(s))\Rr) = \dot\Psi(s)\cdot \nabla \Ll(b\cdot\nabla\xi\Rr) \Ll(\Psi(s)\Rr) >0
\end{align*}
where in the last inequality we used~\eqref{e.cond_xi} and $\dot\Psi(s)\in\S^\D_{++}$ due to $\Psi\in \Pi_\mathrm{reg}(z)$. The above display implies (by diagonalization) $\gamma(s)\in \S^\D_{++}$ for every $s$ and thus \eqref{e.gamma>0}.

Part~\eqref{i.l.weak->gamma>0_1}.
Let $\Psi$ satisfy \eqref{e.gamma>0}.
Write $\bar \mu(s)=\mu(t)-\mu(s)$.
Due to~\eqref{e.gamma>0} and $\dot\mu =\gamma$, there is $c>0$ such that $\bar\mu(r)\geq c\identity$ for $r$ sufficiently close to $s$. Using this and the boundedness of $\gamma$ (due to \ref{i.xi_loc_lip} and $\Psi\in\Pi_\mathrm{Lip}(z)$), we have
\begin{align*}
    \Ll|\sqrt{\bar\mu(r)}-\sqrt{\bar\mu(r')}\Rr| = \Ll|\Ll(\sqrt{\bar\mu(r)}-\sqrt{\bar\mu(r')}\Rr)\Ll(\sqrt{\bar\mu(r)}+\sqrt{\bar\mu(r')}\Rr)\Ll(\sqrt{\bar\mu(r)}+\sqrt{\bar\mu(r')}\Rr)^{-1}\Rr|
    \\
    = \Ll|\Ll(\bar\mu(r)-\bar\mu(r')\Rr)\Ll(\sqrt{\bar\mu(r)}+\sqrt{\bar\mu(r')}\Rr)^{-1}\Rr| \leq C|r-r'|
\end{align*}
for some constant $C$ and for $r,\,r'$ sufficiently close to $s$.
For $h\in \S^\D_{++}$ and $a \in \S^\D$,
\begin{align*}
    \mathrm{D}_{\sqrt{h}}(a) = \lim_{\eps\to0}\eps^{-1}\Ll(\sqrt{h+\eps a}-\sqrt{h}\Rr)
\end{align*}
exists.
Using these and $\bar\mu(s+\eps) = \bar\mu(s)-\eps\gamma(s)+o(\eps)$, we get
\begin{align*}
    \sqrt{\bar\mu(s+\eps)}-\sqrt{\bar\mu(s)} = \sqrt{\bar\mu(s)-\eps\gamma(s)+o(\eps)}-\sqrt{\bar\mu(s)} = \sqrt{\bar\mu(s)-\eps\gamma(s)}-\sqrt{\bar\mu(s)} + o(\eps)
    \\
    = \eps \mathrm{D}_{\sqrt{\bar\mu(s)}}\Ll(-\gamma(s)\Rr)+o(\eps).
\end{align*}
Hence, $\sqrt{\bar\mu(s)}$ is differentiable and $\Psi$ satisfies \eqref{e.weak_gamma>0}.
\end{proof}

We want to show a key identity to prove Theorem~\ref{t.lim}. We need two lemmas as preparation.
We generalize the definition of $\Pi(z)$ in \eqref{e.Pi(z)} by setting $\Pi(a) = \{\pi\in\Pi:\pi(1)=a\}$ for $a\in\S^\D_+$.

We recall the Lipschitzness of $\pi\mapsto \sP(\pi)$ (e.g.\ see \cite[Lemma~3.4]{chen2023self} with $x=0$ therein).

\begin{lemma}\label{l.Lip_P(pi,x)}
For every $R>0$, there is a constant $C>0$ such that
\begin{align*}
    \Ll|\sP(\pi)-\sP(\pi')\Rr|\leq C\int_0^1 \Ll|\pi(s)-\pi'(s)\Rr| \d s,\quad\forall \pi,\,\pi'\in \bigcup_{a\in\S^\D_+:\:|a|\leq R}\Pi(a).
\end{align*}
\end{lemma}

\begin{proof}
By \cite[Lemma~3.4]{chen2023self} (with $x=0$ therein),
\begin{align*}
    \Ll|\sP(\pi)-\sP(\pi')\Rr|\leq \frac{1}{2}\int_0^1 \Ll(\Ll|\nabla\xi\circ\pi(s)-\nabla\xi\circ\pi'(s)\Rr| + \Ll|\theta(\pi(s))-\theta\Ll(\pi'\Ll(s\Rr)\Rr)\Rr| \Rr) \d s.
\end{align*}
We can verify that there is a constant $c>0$ such that $\sup_{s\in[0,1]}|\pi(s)|\leq c$ for every $\pi\in\cup_{a\in\S^\D_+:\:|a|\leq R}\Pi(a)$. Then, the desired bound follows from the local Lipschitzness of $\nabla\xi$ and $\theta$ due to \ref{i.xi_loc_lip}.
\end{proof}

\begin{lemma}\label{l.infP>infP}
Let $z\in\S^\D_{++}$. As $\eps \searrow 0$, $\inf_{\pi\in\Pi(z)} \sP(\pi) \geq \inf_{\pi\in\Pi(z-\eps\identity_\D)} \sP(\pi) + o_\eps(1)$.
\end{lemma}
\begin{proof}
This is a consequence of \cite[Lemma~3.6]{chen2023on}. To apply this lemma, we set $y=0$ and substitute $(z-\eps \identity_\D,z)$ for $(z,z')$ therein. Then, we get $\inf_{\Pi(z)} \sP -\inf_{\Pi(z-\eps\identity_\D)} \sP \geq - CK(z-\eps \identity_\D,|\eps\identity_\D|)$ for some constant $C$ and $K$ defined in \cite[(3.8)]{chen2023self}. To show this vanishes as $\eps\to0$, we need $z-\eps \identity_\D \geq \delta\identity_\D$ for some $\delta>0$ and all sufficiently small $\eps$, which is a result of $z\in \S^\D_{++}$.
\end{proof}

Now, we are ready to prove the aforementioned key identity.

\begin{lemma}\label{l.limF=limP}
Let $z\in\S^\D_{++}$, let $P_1$ satisfy~\eqref{e.supp_P_1} and let $\xi$ be twice continuously differentiable and satisfy~\eqref{e.cond_xi}. Then,
\begin{align*}
    \inf_{\Psi\in \Pi_\mathrm{reg}(z)}\inf_{\alpha\in\cM} \mathscr{F}(\Psi,\alpha) = \inf_{\pi\in\Pi(z)}\sP(\pi).
\end{align*}
\end{lemma}
\begin{proof}
By Lemma~\ref{l.weak->gamma>0} and Corollary~\ref{c.well-defined}, the Parisi PDE solution $\Phi_{\tilde P_1^z,\Psi,\alpha}$ exists for every $\Psi\in\Pi_\mathrm{reg}(z)$.
We denote the left-hand (resp.\ right-hand) side of the announced identity by $\lhs$ (resp.\ $\rhs$).
Using Lemma~\ref{l.Lipschitz_phi}, Lemma~\ref{l.Lipschitz_pi}, and the local Lipschitzness of $\theta$ (due to \ref{i.xi_loc_lip}), we can replace $\cM$ by $\cM_\d$ and $\Pi(z)$ by $\Pi_\d(z)$, namely,
\begin{align*}
    \lhs = \inf_{\Psi\in \Pi_\mathrm{Lip}(z)}\inf_{\alpha\in\cM_\d} \mathscr{F}(\Psi,\alpha),\qquad \rhs = \inf_{\pi\in\Pi_\d(z)}\sP(\pi).
\end{align*}
For every $\Psi\in\Pi_\mathrm{Lip}(z)$ and $\alpha \in \cM_\d$, we can apply Lemma~\ref{l.sF=sP} to get $\lhs\geq\rhs$. 

Now, we prove the other direction. 
Lemma~\ref{l.infP>infP} gives
\begin{align*}
    \rhs \geq \inf_{\pi\in \Pi(z-\eps\identity_\D)}\sP(\pi) + o_\eps(1)
\end{align*}
as $\eps\to0$. For each $\eps$, we can choose some $\pi_\eps \in \Pi(z-\eps\identity_\D)$ such that $\rhs \geq \sP(\pi_\eps)+o_\eps(1)$. We set $\pi_\eps(s)= \pi_\eps(0)$ for $s<0$ and $\pi_\eps(s)=\pi_\eps(1)$ for $s>1$. Then, we mollify $\pi_\eps$ by setting $\tilde \pi_\eps(s) = \int \varphi(r)\pi_\eps(s-r)\d r$ where $\varphi$ is a real-valued nonnegative smooth function supported on a compact set in $(-\infty,0)$. 
Due to the choice of $\supp \varphi$, we have $\tilde\pi_\eps(1)=\pi_\eps(1)$. Hence, it is easy to see $\tilde \pi_\eps \in \Pi(z-\eps \identity_\D)$ and that $\tilde \pi_\eps$ is continuously differentiable. 
In view of Lemma~\ref{l.Lip_P(pi,x)}, we can choose $\varphi$ (depending on $\eps$) to satisfy $\rhs \geq \sP(\tilde \pi_\eps)+o_\eps(1)$.
Lastly, we set $\hat \pi_\eps(s) = \tilde\pi_\eps + \eps s \identity_\D$. Then, we have $\hat \pi_\eps(1)=z$ and $\frac{\d}{\d s}\hat \pi_\eps(s) \in \S^\D_{++}$ for every $s$. Therefore, $\hat \pi_\eps\in\Pi_\mathrm{reg}(z)$. Moreover, by Lemma~\ref{l.Lip_P(pi,x)}, we have 
\begin{align*}
    \rhs \geq \sP\Ll(\hat\pi_\eps\Rr)+o_\eps(1).
\end{align*}
Now, we take $\Psi=\hat\pi_\eps$, $\alpha(s)=s$ for $s\in[0,1]$, and a sequence $(\alpha_n)_{n\in\N}$ in $\cM_\d$ that satisfies $\alpha_n^{-1}$ converges to $\alpha^{-1}$ uniformly. By Lemmas~\ref{l.Lip_P(pi,x)} and~\ref{l.sF=sP}, we have
\begin{align*}
    \lim_{n\to\infty}\sF(\Psi,\alpha_n) = \sP\Ll(\Psi\circ\alpha^{-1}\Rr) = \sP\Ll(\hat\pi_\eps\Rr).
\end{align*}
The above two displays together imply $\rhs \geq \lhs$, finishing the proof.
\end{proof}

\begin{proof}[Proof of Theorem~\ref{t.lim}]
Recall that the Parisi formula~\eqref{e.parisi} was proved in \cite[Theorem~1.1]{chen2023on} provided that $\xi$ is convex on $\S^\D_+$. Then, Theorem~\ref{t.lim} follows from~\eqref{e.parisi} and Lemma~\ref{l.limF=limP}.
\end{proof}

\begin{proof}[Proof of Theorem~\ref{t.convex}]
In view of the expression of $\sF(\Psi,\alpha)$ in~\eqref{e.sF}, it suffices to show that if $\alpha_0$ and $\alpha_1$ are distinct, then for every $x\in \R^\D$ and every $\lambda\in (0,1)$
\begin{align*}
    \Phi_{\tilde P_1^z,\Psi, (1-\lambda)\alpha_0 + \lambda \alpha_1}(0,x) < (1-\lambda)\Phi_{\tilde P_1^z,\Psi, \alpha_0}(0,x) + \lambda\Phi_{\tilde P_1^z,\Psi, \alpha_1}(0,x).
\end{align*}
Recall that in Section~\ref{s.convex}, we use the notation $\Phi_\alpha = \Phi_{P_1,\Psi,\alpha}$ introduced in~\eqref{e.Psi_alpha=Psi_PPsialpha}. 
Note that the condition~\eqref{e.not_Dirac} for $P_1$ is inherited by $\tilde P_1^z$. 
Applying Proposition~\ref{p.strict_convex} with $\tilde P_1^z$ substituted for $P_1$, we get the strict convexity in the above display and thus part~\eqref{i.Phi_convex} of the theorem.

We identify $\alpha\in\cM$ with the probability measure it represents. Lemma~\ref{l.Lipschitz_phi} implies that $\alpha\mapsto \sF(\Psi,\alpha)$ is continuous in the topology of weak convergence of probability measures. Since $\cM$ consists of probability measures on $[0,1]$, it is tight. Hence, by Prokhorov's theorem, minimizers of $\inf_{\alpha\in\cM} \sF(\Psi,\alpha)$ exist. The uniqueness follows from part~\eqref{i.Phi_convex}. Part~\eqref{i.uniq_minimizer} is thus verified.
\end{proof}

\subsection{One-dimensional case}

\begin{theorem}\label{t.D=1}
Let $\D=1$. Assume that $\xi:\R\to\R$ is twice differentiable on $[0,z]$ and that its second-order derivative $\xi''(s)>0$ a.e.\ on $[0,z]$. Let $\Psi^\star:[0,1]\to [0,\infty)$ be given by $\Psi^\star(s)= zs$. If $\xi$ is convex on $\R_+$, then
\begin{align}\label{e.D=1limF_N}
    \lim_{N\to\infty} F_N = \inf_{\alpha\in\cM} \sF(\Psi^\star,\alpha).
\end{align}
If $z>0$, then $\alpha\mapsto \sF(\Psi^\star,\alpha)$ is strictly convex
and there exists a unique minimizer $\alpha^\star$.
\end{theorem}

\begin{proof}
If $z=0$, then $\Psi^\star$ is constantly zero and thus condition~\eqref{e.weak_gamma>0} is satisfied. Hence, by Corollary~\ref{c.well-defined}, $\sF(\Psi^\star,\alpha)$ is well-defined. Note that $\sF(\Psi^\star,\alpha)$ is independent of $\alpha$ and $\inf_{\pi\in\Pi(z)}\sP(\pi)= \sP(\pi_0)$ where $\pi_0$ is constantly zero. Due to $\pi_0 = \Psi^\star\circ\alpha^{-1}$ for every $\alpha$, we obtain \eqref{e.D=1limF_N} from Lemma~\ref{l.sF=sP} and the Parisi formula~\eqref{e.parisi} (given that $\xi$ is convex on $\S^1_+=\R_+$). Note that there is no need to find a minimizer in this degenerate case.

Now, we assume $z>0$.
Since $\Psi^\star\in \Pi_\mathrm{reg}(z)$ and $\xi$ satisfies~\eqref{e.cond_xi}, Lemma~\ref{l.weak->gamma>0} implies that \eqref{e.weak_gamma>0} is satisfied. Hence, Corollary~\ref{c.well-defined} ensures that $\sF(\Psi^\star,\alpha)$ is well-defined.
Arguing as in the beginning of the proof of Lemma~\ref{l.limF=limP}, we have
\begin{align*}
    \inf_{\alpha\in\cM} \sF(\Psi^\star,\alpha) = \inf_{\alpha\in\cM_\d } \sF(\Psi^\star,\alpha),\qquad \inf_{\pi\in \Pi(z)}\sP(\pi) = \inf_{\pi\in \Pi_\d(z)}\sP(\pi).
\end{align*}
Lemma~\ref{l.sF=sP} implies $\inf_{\alpha\in\cM_\d } \sF(\Psi^\star,\alpha) \geq \inf_{\pi\in \Pi_\d(z)}\sP(\pi)$. On the other hand, for every $\pi\in \Pi_\d(z)$, we choose $\alpha$ to satisfy $\alpha^{-1} = \frac{1}{z}\pi$. Then, $\Psi\circ\alpha^{-1} = \pi$. From this and Lemma~\ref{l.sF=sP}, we get $\inf_{\alpha\in\cM_\d } \sF(\Psi^\star,\alpha) \leq \inf_{\pi\in \Pi_\d(z)}\sP(\pi)$. Hence, \eqref{e.D=1limF_N} follows from~\eqref{e.parisi} provided that $\xi$ is convex on $\R_+$.

Substituting $\Psi^\star$ for $\Psi$ in the definition of $\mu$ in \eqref{e.mu}, we get $\gamma(s) = z\xi''(zs)$ for $s\in[0,1]$. The assumption on $\xi''$ verifies \eqref{e.gamma>0}. Then, the uniqueness of the minimizer follows from Theorem~\ref{t.convex}.
\end{proof}

Recall the Parisi PDE \eqref{e.Parisi_PDE} in the general case.
Under the setting in Theorem~\ref{t.D=1}, $\Phi_{\tilde P_1^z,\Psi^\star,\alpha}$ in $\sF(\Psi^\star,\alpha)$ solves (in the sense of Definition~\ref{d.sol_pde})
\begin{align*}
    \partial_s \Phi(s,x) + \frac{1}{2}z\xi''(zs)\Ll(\partial^2_x \Phi(s,x) + \alpha(s) \Ll(\partial_x \Phi(s,x)\Rr)^2\Rr) =0,\quad \text{on $[0,1]\times \R$}.
\end{align*}
By Proposition~\ref{p.reg_Phi}~\eqref{i.dsdxPhi_exists}, for continuous $\alpha$, $\Phi_{\tilde P_1^z,\Psi^\star,\alpha}$ is the classical solution of this PDE.

\section{Potts spin glasses}\label{s.Potts}

We describe the setting of the Potts spin glass with symmetric mixed $p$-spin interactions.
Let $D\geq 2$,
let $\basis = \{e_1,\dots,e_\D\}$ be the standard basis of $\R^\D$, and let $P_1$ be the uniform probability measure on $\basis$. 
For $\R^{\D\times N}$-valued spin configuration $\sigma$ sampled from $P_N$, we denote by $\sigma_{k\bullet} = (\sigma_{ki})_{i=1}^N$ its $k$-th row vector, for $k\in\{1,\dots,\D\}$. For $N\in\N$, $k\in\{1,\dots,\D\}$, and  $p\in \N\cap[ 2,\infty)$, we consider
\begin{align*}
    H_{N,p}(\sigma_{k\bullet}) = \frac{1}{N^{\frac{p-1}{2}}}\sum_{i_1,\dots,i_p=1}^N g_{i_1,\dots,i_p}\sigma_{ki_1}\cdots \sigma_{ki_p}
\end{align*}
where $(g_{i_1,\dots,i_p})_{1\leq i_1,\dots,i_p\leq N;\; p\geq 2}$ is a family of independent standard Gaussians.
Given a sequence $(\beta_p)_{p= 2}^\infty$ of nonnegative numbers where at least one of them is positive, we set
\begin{align*}
    H_N(\sigma_{k\bullet}) = \sum_{p=2}^\infty \beta_p H_{N,p}(\sigma_{k\bullet}).
\end{align*}
We assume that $(\beta_p)_{p= 2}^\infty$ decays fast enough so that the above is well-defined. Finally, we define
\begin{align*}
    H_N(\sigma) = \sum_{k=1}^\D H_N(\sigma_{k\bullet}).
\end{align*}
Note that the law of $H_N(\sigma)$ is invariant under any permutation of the labels $\{1,\dots,\D\}$. 
Recall the definition of $\xi$ in~\eqref{e.xi}. Under this setting, we can verify
\begin{align}\label{e.xi_potts}
    \xi(a) = \sum_{k,k'=1}^\D\sum_{p=2}^\infty \beta^2_p a^p_{kk'},\quad\forall a\in\R^{\D\times\D}.
\end{align}
In Section~\ref{s.Psi_potts}, we will explain that the symmetry heuristically implies that $\lim_{N\to \infty} F_N$ is equal to $\inf_{\alpha\in\cM} \sF(\Psi^\star,\alpha)$ for $\Psi^\star$ given by
\begin{align}\label{e.Psi_potts}
    \Psi^\star(s) = \frac{s}{\D}\identity_\D + \frac{1-s}{\D^2}\one_\D,\quad\forall s\in[0,1]
\end{align}
where $\one_\D$ is the $\D\times\D$ matrix with all entries equal to $1$.
Relevantly, it has been shown in \cite[Corollary~1.2]{chen2023self} that $z=\frac{1}{D}\identity_\D$ for $z$ in~\eqref{e.z} and indeed we have $\Psi^\star(1)=z$ as expected.

The goal of this section is to verify the results mentioned in Section~\ref{s.discussion}.

\begin{theorem}\label{t.convex_Potts}
For $P_1$ uniform on $\basis$, $\xi$ in~\eqref{e.xi_potts}, $z=\frac{1}{\D}\identity_\D$, and $\Psi=\Psi^\star$ in~\eqref{e.Psi_potts}, statements in parts~\eqref{i.Phi_convex} and~\eqref{i.uniq_minimizer} of Theorem~\ref{t.convex} hold.
\end{theorem}

We will prove the theorem in two cases:
\begin{enumerate}
    \item \label{i.potts_case_1} there is $p\geq 3$ such that $\beta_p>0$;
    \item \label{i.potts_case_2} $\beta_2>0$ and $\beta_p=0$ for all $p\geq 3$.
\end{enumerate}

\subsection{Expected Lipschitz path}\label{s.Psi_potts}

We explain~\eqref{e.Psi_potts} only heuristically.
Denote by $\Sym$ the permutation group of $\{1,\dots,\D\}$ consisting of bijections $\ks:\{1,\dots,\D\}\to \{1,\dots,\D\}$.
We have $\sigma\stackrel{\d}{=}\sigma^\ks = \Ll(\sigma_{\ks(k),i}\Rr)_{1\leq k\leq\D,\, 1\leq i\leq N}$ under $P_N$ for every $\ks\in\Sym$.
From~\eqref{e.xi_potts}, we can see $(H_N(\sigma))_{\sigma\in \R^{\D\times N}}\stackrel{\d}{=}(H_N(\sigma^\ks))_{\sigma\in \R^{\D\times N}}$ under $\E$.
Hence, we can see that $R_N = \frac{\sigma\sigma'^\intercal}{N}$ and $R^\ks_N=\frac{\sigma^\ks (\sigma^\ks)'^\intercal }{N}$ have the same distribution under $\E\la\cdot\ra$ where $\la\cdot\ra$ is the Gibbs measure associated with $F_N$. 

Let us assume that $R_N$ (resp.\ $R^\ks_N$) converges in law to some $R_\infty$ (resp.\ $R^\ks_\infty$) as $N\to\infty$. 
Assuming that the limit of the overlap array satisfies the Ghirlanda--Guerra identities, which allows us to apply Panchenko's synchronization (\cite[Theorem~3]{pan.potts} and \cite[Theorem~4]{pan.vec}).
Let $\mathscr{S}:[0,\infty)\to \S^\D_+$ (resp.\ $\mathscr{S}^\ks$) be the synchronization map associated with $R_\infty$ (resp.\ $R^\ks_\infty$) and we have, almost surely,
\begin{align}\label{e.sync}
    R_\infty=\mathscr{S}(\tr(R_\infty)),\qquad R^\ks_\infty=\mathscr{S}^\ks(\tr(R^\ks_\infty)).\end{align}
Since the synchronization map only depends on the distribution of the overlap, we must have $\mathscr{S} =\mathscr{S}^\ks$. 
This along with the observation that $\tr(R_\infty)=\tr(R^\ks_\infty)$ and the second identity in~\eqref{e.sync} implies $R^\ks_\infty = \mathscr{S}(\tr(R_\infty))$. Comparing this with the first equality in~\eqref{e.sync}, we should expect
\begin{align}\label{e.symmetry_S}
    \mathscr{S}_{kk'} \Ll(\kappa^{-1}(s)\Rr)=\mathscr{S}_{\ks(k)\ks(k')}\Ll(\kappa^{-1}(s)\Rr),\quad\forall k,\, k'\in\{1,\dots,\D\},\ s \in [0,1],\ \ks\in\Sym
\end{align}
where $\kappa:\R\to[0,1]$ is the cumulative distribution function of $\tr(R_\infty)$.

Setting $\pi = \mathscr{S}\circ\kappa^{-1}$, we have $R_\infty \stackrel{\d}{=}\pi(U)$ where $U$ is the uniform random variable on $[0,1]$. Hence, $\pi$ describes the limit of the overlap and thus the Parisi formula~\eqref{e.parisi} (when $\xi$ is convex on $\S^\D_+$) is expected to achieve its minimum at $\pi$.

Let us relate $\pi$ to $\Psi^\star$.
Varying $\ks$, we deduce from \eqref{e.symmetry_S} that $\pi_{kk} = \pi_{k'k'}$ and $\pi_{kk'}= \pi_{k''k'''}$ for all distinct $k,k',k'',k'''$. Namely, $\pi$ has the same diagonal entries and the same off-diagonal entries.
Due to $\sum_{k,k'=1}^\D(R_N)_{kk'}=1$ (as a result of the support $P_1$ on $\Sigma$), we get $\sum_{kk'}\pi_{kk'}=1$ and thus $\pi_{kk}+(D-1)\pi_{k'k''}=1$ for distinct $k,k',k''$.
Due $(R_N)_{kk'}\geq 0$ for every $k,k'$, all entries of $\pi$ are nonnegative. By these and $\pi(s)\in\S^\D_+$ for all $s$, the range of $\pi$ must be a subset of the range of $\Psi^\star$. Since both $\pi$ and $\Psi^\star$ are increasing and left-continuous, we can find a left-continuous and increasing map $\varphi:[0,1]\to[0,1]$ such that $\Psi^\star \circ\varphi = \pi$.

Finally, choose $\alpha\in\cM$ to satisfy $\alpha^{-1}=\psi$.
In view of Lemma~\ref{l.sF=sP} and the Parisi formula~\eqref{e.parisi}, we expect $\lim_{N\to\infty} F_N$ to be $\inf_{\alpha\in\cM}\sF(\Psi^\star,\alpha)$ if $\xi$ is convex on $\S^\D_+$.

\subsection{\texorpdfstring{Case~\eqref{i.potts_case_1}}{Case~1}}

Recall $\mu$ defined in \eqref{e.mu}.
\begin{lemma}\label{l.cond_ok_potts}
Let $\xi$ and $\Psi^\star$ be given in~\eqref{e.xi_potts} and~\eqref{e.Psi_potts} respectively. In case~\eqref{i.potts_case_1}, $\Psi^\star$ satisfies~\eqref{e.gamma>0}.
\end{lemma}

\begin{proof}
We write $\Psi=\Psi^\star$.
Using~\eqref{e.xi_potts}, we can compute
\begin{align}\label{e.nabla_xi=}
    \nabla\xi(a) = \sum_{p=2}^\infty p\beta^2_p\Ll(a^{p-1}_{kk'}\Rr)_{1\leq k,k'\leq \D} = \sum_{p=2}^\infty p\beta^2_p a^{\circ (p-1)},\quad \forall a \in \R^{\D\times \D},
\end{align}
where $\circ$ denotes the Hadamard product, namely, $a\circ b = (a_{kk'}b_{kk'})_{1\leq k,k'\leq \D}$ for $a,b\in\R^{\D\times \D}$ (not to confuse with the composition of functions).
Then, we can compute
\begin{align}\label{e.gamma=_potts}
    \gamma(s) = \sum_{p=2}^\infty p(p-1)\beta_p^2 \Psi(s)^{\circ(p-2)}\circ \dot\Psi(s),\quad\forall s \in[0,1].
\end{align}
where
\begin{align}\label{e.dotPhi}
    \dot\Psi(s)  = \frac{1}{\D}\identity_\D -\frac{1}{\D^2}\one_\D
\end{align}
which is constant.
Due to $\one_\D\in\S^\D_+$, we have $\Psi(s)\cgeq \frac{s}{\D}\identity_\D$ and thus
\begin{align}\label{e.Psi>0}
    \Psi(s)\in\S^\D_{++},\quad\forall s \in(0,1].
\end{align}
For distinct $k$ and $k'$, we can verify
\begin{align*}
    \Ll(\Psi(s)\circ\dot\Psi(s)\Rr)_{kk'} &= \Psi(s)_{kk'}\circ\dot\Psi(s)_{kk'}=\frac{1-s}{\D^2}\dot\Psi(s)_{kk'},
    \\
    \Ll(\Psi(s)\circ\dot\Psi(s)\Rr)_{kk} &= \Psi(s)_{kk}\circ\dot\Psi(s)_{kk} = \frac{s}{D}\Ll(\frac{1}{\D}-\frac{1}{\D^2}\Rr)+\frac{1-s}{\D^2}\dot\Psi(s)_{kk}.
\end{align*}
Therefore,
\begin{align}\label{e.PsiPsi'=}
    \Psi(s)\circ\dot\Psi(s) = \frac{s(\D-1)}{\D^3}\identity_\D + \frac{1-s}{\D^2}\dot\Psi(s).
\end{align}
Note that by Jensen's inequality,
\begin{align}\label{e.xdotPsix>0}
    x^\intercal \dot\Psi(s) x = \sum_{k=1}^\D\frac{1}{\D}x_k^2 - \Ll(\sum_{k=1}^\D\frac{1}{\D}x_k\Rr)^2 \geq 0,\quad\forall x\in\R^\D,
\end{align}
which implies $\dot\Psi(s)\in\S^\D_+$ for all $s$. Recall $D\geq 2$. Hence,~\eqref{e.PsiPsi'=} yields
\begin{align}\label{e.PsiPsi'>0}
    \Psi(s)\circ\dot\Psi(s)\in\S^\D_{++},\quad\forall s\in(0,1].
\end{align}
Let $p\geq 3$ satisfy $\beta_{p}>0$. Recall that the Schur product theorem \cite[Theorem~7.5.3]{horn2012matrix} states that the Hadamard product of two positive definite (resp.\ semi-definite) matrices is positive definite (resp.\ semi-definite). Using this,~\eqref{e.gamma=_potts},~\eqref{e.Psi>0}, and~\eqref{e.PsiPsi'>0}, we get
\begin{align*}
    \gamma(s)\cgeq p(p-1)\beta_{p}^2\Psi(s)^{\circ(p-3)}\circ\Ll(\Psi(s)\circ\dot\Psi(s)\Rr)\in\S^\D_{++},\quad\forall s\in(0,1],
\end{align*}
which verifies~\eqref{e.gamma>0}.
\end{proof}

\begin{proof}[Proof of Theorem~\ref{t.convex_Potts} in case~\eqref{i.potts_case_1}]
It follows from Lemma~\ref{l.cond_ok_potts} and Theorem~\ref{t.convex}. 
\end{proof}

\subsection{\texorpdfstring{Case~\eqref{i.potts_case_2}}{Case~2}}

We show that~\eqref{e.gamma>0} does not hold in this case. By~\eqref{e.gamma=_potts} and~\eqref{e.dotPhi}, we get
\begin{align}\label{e.gamma=Potts2}
    \gamma(s) = 2\beta^2_2 \dot\Psi^\star(s) = \frac{2\beta^2_2}{\D^2}\Ll(\D\identity_\D - \one_\D\Rr)
\end{align}
which is not positive definite because 
\begin{align}\label{e.w}
    \gamma(s)w =0,\quad \text{for } w = (1,1,\dots,1)\in\R^\D,
\end{align}
due to $\Ll(\D\identity_\D - \one_\D\Rr) w=0$.
So, Theorem~\ref{t.convex} is not applicable.
Despite the lack of~\eqref{e.gamma>0}, we can exploit relations satisfied by $\gamma$ to retain the strict convexity.

As commented in Remark~\ref{r.condition_PDE}, we assume that $\Psi$ satisfies~\eqref{e.weak_gamma>0} instead of~\eqref{e.gamma>0} in Section~\ref{s.PDE}. The next lemma makes the results in Section~\ref{s.PDE} still available. In particular, as a result of Corollary~\ref{c.well-defined}, the Parisi PDE solutions are well-defined for $\Psi^\star$.

\begin{lemma}\label{l.weak_cond_Potts2}
In case~\eqref{i.potts_case_2}, $\Psi^\star$ satisfies~\eqref{e.weak_gamma>0}.
\end{lemma}
\begin{proof}
Recall $\mu$ defined in \eqref{e.mu}. Using \eqref{e.Psi_potts},  \eqref{e.nabla_xi=}, and \eqref{e.dotPhi}, we have $\mu(s) = 2\beta^2_2\Psi^\star(s)$ and $\mu(t)-\mu(s)= 2\beta^2_2(t-s)\dot\Psi^\star(0)$. Since \eqref{e.xdotPsix>0} verifies $\dot\Psi^\star(0)\in\S^\D_+$, we get $\sqrt{\mu(t)-\mu(s)}= \sqrt{2(t-s)}\beta_2\sqrt{\dot\Psi^\star(0)}$ which is differentiable in $s$ if $s<t$.
\end{proof}

Since $\gamma(s)$ is constant in this case, we simply write $\gamma=\gamma(s)$ for all $s\in[0,1]$ henceforth.
Recall that $\{e_1,\dots,e_\D\}$ is the standard basis.
For $w$ in~\eqref{e.w}, we set 
\begin{align}\label{e.v_k}
    v_k = De_k - w,\quad\forall k\in\{1,\dots,\D\}.
\end{align}
Using $\sum_{k=1}^\D e_k = w$, we can compute
\begin{align*}
    \sum_{k=1}^\D v_kv_k^\intercal = \sum_{k=1}^\D \Ll(\D^2 e_ke_k^\intercal - \D e_kw^\intercal - \D we_k^\intercal +ww^\intercal\Rr) = \D^2\identity_\D - \D \one_\D.
\end{align*}
Therefore,~\eqref{e.gamma=Potts2} becomes
\begin{align}\label{e.gamma=sumvv}
    \gamma = \frac{2\beta^2_2}{\D^3}\sum_{k=1}^\D v_kv_k^\intercal.
\end{align}
Using $\one_\D^2 = \D \one_\D$, we also have $\Ll(\D \identity_\D - \one_\D\Rr)^2=\D\Ll(\D \identity_\D - \one_\D\Rr)$, which together with~\eqref{e.gamma=Potts2} implies
\begin{align}\label{e.gamma^2=cgamma}
    \gamma^2 = \frac{2\beta^2_2}{\D}\gamma.
\end{align}
The two identities,~\eqref{e.gamma=sumvv} and~\eqref{e.gamma^2=cgamma}, are crucial to our proof of strict convexity.
We start with some lemmas.

\begin{lemma}\label{l.yphiy>0_potts}
Let $w$ be given in~\eqref{e.w} and let be $\phi$ given in~\eqref{e.phi=initial}. Then, $\phi$ is convex and the following holds at every $x\in\R^\D$: for every $y\in\R^\D$, $y^\intercal\nabla^2\phi(x) y=0$ if and only if $y=r w$ for some $r\in\R$.
\end{lemma}
\begin{proof}
In this setting, $\phi(x) = \log \sum_{k=1}^\D\frac{1}{\D}e^{x_k}$. For every $y\in\R^\D$, we can compute
\begin{align*}
    y^\intercal\nabla^2\phi(x)y=\frac{\d^2}{\d \eps^2}\phi(x+\eps y)\Big|_{\eps=0} = \frac{\sum_{k=1}^\D y_k^2e^{x_k}}{\sum_{k=1}^\D e^{x_k}} - \Ll(\frac{\sum_{k=1}^\D y_ke^{x_k}}{\sum_{k=1}^\D e^{x_k}}\Rr)^2.
\end{align*}
By Jensen's inequality, the right-hand side is nonnegative. It is zero if and only if $y_k$'s are equal. 
\end{proof}

We need a version of Lemma~\ref{l.Phi_convex}~\eqref{i.d^2Phi>0}.
\begin{lemma}\label{l.vnabla^2Phiv>0}
Let $(v_k)_{1\leq k\leq \D}$ be given in~\eqref{e.v_k}.
For every $\alpha\in\cM$, $(s,x)\in[0,1]\times\R^\D$, and $k\in\{1,\dots,\D\}$, 
\begin{align*}
    v^\intercal_k \nabla^2\Phi_{P_1,\Psi,\alpha}(s,x) v_k >0
\end{align*}
\end{lemma}
\begin{proof}
By Lemma~\ref{l.yphiy>0_potts}, we have $v_k^\intercal\nabla^2\phi(x)v_k>0$ for every $x$. Using this and substituting $v_k$ for $y$ in~\eqref{e.ynabla^2Phiy>}, we obtain the desired result.
\end{proof}

\begin{lemma}\label{l.phi(x_lambda)<}
For $x_0,x_1\in\R^\D$, if $|\gamma(x_1-x_0)|>0$, then $\phi(x_\lambda) < (1-\lambda)\phi(x_0) + \lambda\phi(x_1)$ for every $\lambda\in(0,1)$ where $x_\lambda = (1-\lambda)x_0 + \lambda x_1$.
\end{lemma}
\begin{proof}
Set $y=x_1-x_0$. Suppose that $y=r w$ for some $r\in\R $. Then, using $v_k^\intercal w = 0$ and~\eqref{e.gamma=sumvv}, we have $\gamma y =0$ contradicting the assumption. Hence, there does not exist $r\in\R$ such that $y=rw$. Then, by Lemma~\ref{l.yphiy>0_potts},
\begin{align*}
    \frac{\d^2}{\d \lambda^2} \phi(x_\lambda) = y^\intercal\nabla^2\phi(x_\lambda) y >0.
\end{align*}
The desired result thus follows.
\end{proof}

We use an argument similar to that in \cite[Section~4]{jagtob} to prove the strict convexity

\begin{proposition}\label{p.Phi_convex_potts}
Let $\alpha_0,\,\alpha_1\in\cM$ be distinct.
For every $x\in\R^\D$ and $\lambda \in(0,1)$, 
\begin{align*}
    \Phi_{P_1,\Psi^\star,\alpha_\lambda}(0,x) < (1-\lambda)\Phi_{P_1,\Psi^\star,\alpha_0}(0,x) + \lambda\Phi_{P_1,\Psi^\star,\alpha_1}(0,x)
\end{align*}
where $\alpha_\lambda = (1-\lambda)\alpha_0 + \lambda\alpha_1$.
\end{proposition}
\begin{proof}
We write $\Phi_\alpha = \Phi_{P_1,\Psi^\star,\alpha}$.
Let $u=u^\star_{\alpha_\lambda}$ be the maximizer given in~\eqref{e.def_u} for $s=0$ and $t=1$. For $\ell \in \{0,\lambda,1\}$, we set
\begin{align*}
    Y_\ell & = x+\int_0^1\alpha_\ell \gamma u \d r + \int_0^1 \sqrt{\gamma}\d B.
\end{align*}
Notice $Y_\lambda = (1-\lambda)Y_0 + \lambda Y_1$.
Assuming
\begin{align}\label{e.P()>0}
    \P\Ll(\Ll|\gamma(Y_1-Y_0)\Rr|>0\Rr)>0
\end{align}
and using Lemma~\ref{l.phi(x_lambda)<} and the convexity of $\phi$ given in Lemma~\ref{l.yphiy>0_potts}, we get
\begin{align*}
    \E \phi(Y_\lambda) < (1-\lambda) \E \phi(Y_0) + \lambda \E \phi(Y_1).
\end{align*}
Using this and Proposition~\ref{p.var} for $s=0$ and $t=1$, we have
\begin{align*}
    &\Phi_{\alpha_\lambda}(0,x) = \E \Ll[\phi(Y_\lambda ) - \frac{1}{2}\int_0^1\alpha_\lambda \la \gamma,u u^\intercal\ra_{\S^\D} \d r\Rr]
    \\
    &< (1-\lambda)\E \Ll[\phi(Y_0 ) - \frac{1}{2}\int_0^1\alpha_0 \la \gamma,uu^\intercal\ra_{\S^\D} \d r\Rr]+ \lambda\E \Ll[\phi(Y_1 ) - \frac{1}{2}\int_0^1\alpha_1 \la \gamma,uu^\intercal\ra_{\S^\D} \d r\Rr]
    \\
    &\leq (1-\lambda)\Phi_{\alpha_0}(0,x)+\lambda\Phi_{\alpha_1}(0,x)
\end{align*}
verifying the strict convexity.

It remains to verify~\eqref{e.P()>0}. It suffices to show $\E|\gamma(Y_1-Y_0)|^2>0$. For brevity, we write $\Delta =\alpha_1 -\alpha_0$. We can compute
\begin{align*}
    \E|\gamma(Y_1-Y_0)|^2 = \E \Ll|\int_0^1\Delta\gamma^2 u\d r\Rr|^2 =\la \gamma^4,\iint_{[0,1]^2}\Delta(t)\Delta(s)\E \Ll[u(t)u(s)^\intercal \Rr]\d t\d s\ra_{\S^\D}
\end{align*}
Let $X=X_{\alpha_\lambda}$ be given in~\eqref{e.SDE} for $s=0$ and $t=1$. By \eqref{e.def_u} and Lemma~\ref{l.deriv_Phi}, we have $u(s) = \nabla\Phi_{\alpha_\lambda}(0,x)+ \int_0^s A\sqrt{\gamma}\d B$ where $A(r) = \nabla^2\Phi_{\alpha_\lambda}(r,X(r))$. 
Then, we have
\begin{align*}
    \E \Ll[u(t)u(s)^\intercal \Rr] = \int_0^{t\wedge s}\E \Ll[ A \gamma A \Rr]\d r
\end{align*}
By~\eqref{e.gamma^2=cgamma} and~\eqref{e.gamma=sumvv}, there are constants $c,c'>0$ depending only on $\beta_2$ and $\D$ such that 
\begin{align*}
    \la \gamma^4, A\gamma A \ra_{\S^\D}  =c \la \gamma, A\gamma A\ra_{\S^\D}  = c'\sum_{k,l=1}^\D \la v_kv_k^\intercal, Av_lv_l^\intercal A \ra_{\S^\D} =  c'\sum_{k,l=1}^\D \tr\Ll(v_kv_k^\intercal Av_lv_l^\intercal A\Rr) 
    \\
    = c' \sum_{k,l=1}^\D \tr\Ll(v_k^\intercal Av_lv_l^\intercal A v_k\Rr)  = c' \sum_{k,l=1}^\D \Ll(v^\intercal_k A v_l\Rr)^2\geq c' \sum_{k=1}^\D \Ll(v^\intercal_k A v_k\Rr)^2 >0
\end{align*}
where the last inequality follows from Lemma~\ref{l.vnabla^2Phiv>0}. Therefore, we have
\begin{align*}
    \la \gamma^4, \E \Ll[u(t)u(s)^\intercal \Rr]\ra_{\S^\D}=\int_0^{t\wedge s} \E \la \gamma^4,A\gamma A\ra_{\S^\D} \d r = p(t\wedge s) = p(t)\wedge p(s)
\end{align*}
for a strictly increasing function $p:[0,\infty)\to[0,\infty)$.
Hence,
\begin{align*}
    \E|\gamma(Y_1-Y_0)|^2 = \iint_{[0,1]^2} \Delta(t)\Delta(s) p(t\wedge s) \d t\d s.
\end{align*}
We can view $(s,t)\mapsto p(t\wedge s)$ as the kernel of a Brownian motion with time change. Since $p$ is strictly increasing, this kernel is strictly positive definite. Due to $\Delta \neq 0$, the right-hand side in the above is strictly positive, which implies~\eqref{e.P()>0} and completes the proof.
\end{proof}

\begin{proof}[Proof of Theorem~\ref{t.convex_Potts} in case~\eqref{i.potts_case_2}]
Recall $z$ in the statement and $\tilde P^z_1$ in \eqref{e.tildeP_1}.
By~\eqref{e.nabla_xi=}, 
\begin{align*}
    \frac{1}{2}\nabla\xi(z)\cdot e_ke_k^\intercal = \beta_2^2z \cdot e_ke_k^\intercal = \frac{\beta^2_2}{\D}
\end{align*}
in this case.
Since $\sigma$ takes value in $\{e_k\}_{k=1}^\D$ under $P_1$, we have $\tilde P_1^z = e^{-\frac{\beta^2_2}{\D}} P_1$. Substituting $\tilde P_1^z$ for $P_1$ in the initial condition $\phi$ in~\eqref{e.phi=initial}, we get $\Phi_{\tilde P_1^z,\Psi^\star,\alpha} = \Phi_{P_1,\Psi^\star,\alpha}- \frac{\beta^2_2}{\D}$. Therefore, Proposition~\ref{p.Phi_convex_potts} implies that $\alpha\mapsto\Phi_{\tilde P_1^z,\Psi^\star,\alpha}(0,x)$ is strictly convex for every $x\in\R^\D$. This yields part~\eqref{i.Phi_convex}. By the same argument in the proof of Theorem~\ref{t.convex}~\eqref{i.uniq_minimizer} in Section~\ref{s.pf_main}, part~\eqref{i.uniq_minimizer} follows.
\end{proof}

\small
\bibliographystyle{abbrv}
\newcommand{\noop}[1]{} \def\cprime{$'$}

\end{document}